\documentclass[10pt,a4paper,reqno]{amsart}
\usepackage{amsmath}
\usepackage{amssymb,latexsym}
\usepackage[dvips]{graphicx}
\usepackage{color}
\usepackage{cite}
\usepackage{amsthm}
\usepackage[notcite,notref]{showkeys}
\usepackage{ulem}
\newcommand{\R}{\mathbb R}

\newcommand{\N}{\mathbb N}

\newcommand{\tendsto}[1]{\renewcommand{\arraystretch}{0.5}
\begin{array}[t]{c}
\longrightarrow \\
{ \scriptstyle #1 }
\end{array}
\renewcommand{\arraystretch}{1}}

\newcommand\cro[1]{\langle #1 \rangle}

\newtheorem{theorem}{Theorem} [section]
\newtheorem{lemma}{Lemma} [section]
\newtheorem{proposition}{Proposition} [section]

\newtheorem{definition}{Definition} [section]

\newtheorem{remark}{Remark}[section]
\newtheorem{hypothesis}{Hypothesis}

\let\ssection=\section\renewcommand{\section}{\setcounter{equation}{0}\ssection}
\begin{document}

\title[Korteweg-de Vries equation with variable coefficients]{On well-posedness for some {K}orteweg-de {V}ries type equations with variable coefficients.}

\subjclass[2010]{P }
\author{Luc Molinet,  Raafat Talhouk and Ibtissame Zaiter}

\address{Luc Molinet, Institut Denis Poisson, Universit\'e de Tours, Universit\'e d'Orl\'eans, CNRS,  Parc Grandmont, 37200 Tours, France.}

\address{Raafat Talhouk, Department of Mathematics, Faculty of Sciences 1 and Laboratory of Mathematics, Doctoral School of Sciences and Technology, Lebanese University Hadat, Lebanon.}
\address{Ibtissame Zaiter, Department of Mathematics, Faculty of Sciences 1 and Laboratory of Mathematics, Doctoral School of Sciences and Technology, Lebanese University Hadat, Lebanon.}

%\thanks{Department of Mathematics, Faculty of Sciences 1 and Laboratory of Mathematics, Doctoral School of Sciences and Technology, Lebanese University Hadat, Lebanon. ({\tt Mohamad.Darwich@lmpt.univ-tours.fr}). }
%\and Samer Israwi%
%\thanks{Department of Mathematics, Faculty of Sciences 1 and Laboratory of Mathematics, Doctoral School of Sciences and Technology, Lebanese University Hadat, Lebanon ({\tt s\_israwi83@hotmail.com})}
%\and Raafat Talhouk%
%\thanks{Department of Mathematics, Faculty of Sciences 1 and Laboratory of Mathematics, Doctoral School of Sciences and Technology, Lebanese University Hadat, Lebanon ({\tt rtalhouk@ul.edu.lb}).} 

\maketitle
\begin{abstract}
In this paper,  KdV-type  equations with time- and space-dependent coefficients are considered. Assuming that the dispersion coefficient in front of $ u_{xxx} $ is positive and  uniformly bounded away from the origin and that a  primitive function of the ratio  between the anti-dissipation  and the dispersion coefficients is bounded from below, we prove the existence and uniqueness of a solution $ u $ such that $ h u $ belongs to a classical Sobolev space, where $ h$ is a function related to this ratio. 
  The LWP in  $ H^s(\R) $, $s>1/2$,  in the classical (Hadamard) sense  is also proven under an assumption on the integrability of this  ratio.  
Our approach combines a  change of unknown with  dispersive estimates. 
Note that previous results were restricted to $ H^s(\R) $, $s>3/2$, and  only used the dispersion to compensate the anti-dissipation and not to lower  the Sobolev index required for well-posedness.

\end{abstract}
%\tableofcontents

%\newpage

\section{Introduction and Main Results}
\subsection{Presentation of the problem}

In this paper, we study  the Cauchy problem for the KdV-type  equation with variable coefficients
\begin{equation}\label{KdV1}
 \left\{
	\begin{array}{l}
      u_t + \alpha(t,x)u_{3x} +\beta(t,x) u_{2x} +\gamma (t,x) u_{x} + \delta (t,x) u \\\\
      \hfill =\epsilon(t,x) u u_x 
    \quad\hbox{ for} \quad(t,x) \in(0,T)\times\R
        %\\\\ \qquad+F(u)u_x+ G(u)u_{xxx}+H(u_x)u_{xx}=f, \quad\hbox{ for} \quad(t,x) \in(0,T]\times\R,
	\\\\
	u_{\vert_{t=0}}=u_0,
	\end{array}\right.
\end{equation}
where $u=u(t,x)$, from $[0,T]\times\R$ into $\R$, is the unknown function of the problem, $u_0=u_0(x)$, from $\R$ into $\R$, is the given  initial condition, 
$\alpha=\alpha(t,x)\ge \alpha_0>0$ $\forall \,(t,x) \in [0,T]\times\R$, 
 and  $\beta,\, \gamma,\, \delta, \, \epsilon$ are real-valued smooth  and bounded given
functions  with exact regularities that  will be precised later. Of course, we will also require a strong condition on the relation between $ \alpha$ and the positive part of $ \beta $.  This equation covers several important unidirectional models for the water waves problems at different regimes which take into account the  variations of the bottom. We have in view in particular the  example of the  KdV equation with variable coefficients (see for instance \cite{david},  \cite{TG}) for which $ \beta\equiv 0 $.  Looking for solutions of  (\ref{KdV1}) plays an important and significant role in the study of unidirectional limits for water wave problems with variable depth and topographies.  %

The study of equations of this type with variable coefficients goes back to the  seminal paper of Craig-Kappeler-Strauss \cite{CKS}  where the local well-posedness (LWP)   in high regularity Sobolev spaces  is established under the condition that $-\beta\ge 0 $.  Actually their results even concern quasilinear version of \eqref{KdV1}. In \cite{A2}, Akhunov proved that the associated linear equation is LWP under an  assumption on the  boundedness uniformly in time  and space of the primitive function   $(t,x)\mapsto \int_0^x r(t,z)dz $  where $r(\cdot,\cdot)$ is   the ratio function $  r(t,z)=\beta(t,z)/\alpha(t,z). $  He also showed some evidences on the   sharpness of this assumption. Adaptation of the LWP in high regularity Sobolev spaces under this hypothesis for quasilinear and fully nonlinear generalizations of \eqref{KdV1} can be found in respectively \cite{A1} and \cite{AAW}. In \cite{IT}, Israwi and the second author proved the LWP of \eqref{KdV1} in 
  $H^s(\R) $, $s>3/2$, under the same type of integrability assumption on the ratio function $ r(t,x) $. Their method of proof uses weighted energy estimates. 

Up to our knowledge, our approach is the first one that enables to treat low regularity solutions. Note that, in sharp contrast to \cite{IT},  we use in a crucial way the dispersive nature of the equation driven by the third order term not only to compensate the anti-diffusion term but also to lower the regularity of the resolution space. We proceed in two steps. In a first step
 we  make a change of unknown  in order to rely the solutions of \eqref{KdV1} to the solutions of the following KdV-type equation with a constant coefficient in front of $u_{3x}$ :
  \begin{equation}\label{KdV2}
 \begin{array}{r}
 u_t + u_{3x} -b(t,x) u_{2x} +c(t,x) u_{x} + d(t,x) u = e(t,x) u u_x +f(t,x) u^2\\\\
     \quad\hbox{ for} \quad(t,x) \in(0,T)\times\R
 \end{array}
\end{equation}
where $ b,c,d,e,f $  are real-valued smooth given functions with this time $ b\ge 0 $. Note that this change of unknown is related to the gauge method that is used in similar contexts as  in  \cite{A2}, \cite{AW}, \cite{IT}. 
 Actually, at this stage, to ensure that the coefficients $ e$ and $ f$ of the nonlinear terms are bounded we will  require the boundedness from above uniformly in $[0,T] \times \R $ of $-\int_0^x r_1(t,z) \, dz $ where $ r_1=\beta_1/\alpha $ is, roughly speaking, the ratio function between the positive part  $\beta_1 $ of $ \beta $ and  $\alpha $ (see Hypothesis \ref{hyp3} in Section \ref{sect3}).  
 
  We then prove that the Cauchy problem associated with \eqref{KdV2} is locally well-posed 
  \footnote{In a forthcoming paper we will show how to enhance the LWP result to $ H^s(\R) $, $s\ge 0 $, that will enable to prove a global well-posedness result for a KdV equation with a variable bottom that is non increasing.}in $ H^s(\R) $, $s>1/2 $, by using the method recently introduced by the first author and S. Vento in \cite{MV} that combines energy's and Bourgain's type estimates. It is worth noticing  that terms as $ c(t,x) u_ x$ and $ -b(t,x) u_{2x}$ may not  be treated by a classical fixed point argument in Bourgain's spaces associated with the KdV linear flow. 
  We would like also to emphasize that we will not require a coercive condition on $ b $ in  $ [0,T]\times \R $ ($ b\ge \beta>0 $ on $ [0,T]\times \R $) but only 
   the non negativity of $ b $. Actually we even obtain the unconditional uniqueness in $ H^s(\R) $ in the case $ b=0 $. 
 
 Coming back to \eqref{KdV1} this proves the existence of a solution $ u $ such that $ hu \in C([0,T];H^s) $ with $ T=T(\|h u_0\|_{H^s}) $, where $ h>0 $ defined in \eqref{defh} is  a function related to the ratio function $ r(\cdot,\cdot) $ (see Theorem \ref{th3}). This solution is the unique solution of \eqref{KdV1} such that $ hu $ belongs in $ L^\infty(0,T;H^s)$. It is worth pointing out that 
  we do not need any assumption (except to be bounded and "smooth") on the  coefficient  $ \beta $ outside a neighborhood  of $ -\infty $. Actually, as noticed in 
   Remark \ref{remark31}, any smooth and bounded $ \beta $ that is non positive uniformly in time at  $ -\infty $ would satisfy our assumption.

 Finally to get the LWP of \eqref{KdV1} in classical Sobolev spaces $H^s(\R) $, $ s>1/2 $, we need not only  $ h$ but also  $1/h $ to be bounded, that corresponds to require $ h $ to be  a classical gauge. This leads to an integrability conditon on $ \R $ uniformly in time of the ratio function $ r_1(\cdot,\cdot) $. 
Note that  this type of condition,  that already appears in other works on the subject as \cite{A2} and \cite{IT},  is proven to be sharp for  the LWP in $H^s(\R) $ of the linear equation in \cite{A2}. 
 In particular, it turns out that anti-diffusion on a compact set will not avoid the local well-posedness of the equation.

To end this introduction, let us recall the linear explanation of this last result that can be found for instance in \cite{AW}. To simplify we concentrate on the linear equation 
$$
u_t +\alpha u_{3x} +\beta u_{2x} = 0 \; .
$$
 and we assume that $\alpha $ and $  \beta $ are constant  on $ [0,T]\times [-R,R] $ with $\alpha>0 $ and $ \beta \ge 0 $. Since a wave packet of amplitude close to $ A$ and frequencies close to $ \xi_0$  moves to the left with a speed close to $\frac{d \omega}{d\xi}(\xi_0)= 3 \alpha \xi_0^2$, this wave packet will stays in $ [-R,R] $ during about 
an interval of time $ \Delta t=  \frac{2R}{\alpha \xi_0^2} $ and thus the effect of the anti-diffusion will make its  amplitude growths to $ A \exp(2  R \frac{\beta}{\alpha}) $ that does not depend on $ \xi_0 $. This shows that the speed of propagation of wave packets induced by the dispersion term of order three $ \partial_x^3 $ is just sufficient to compensate the growth of the amplitude of this wave packet induced by the anti-diffusion on a compact set.
\subsection{Main results}
 In the sequel $ [s] $ denotes the integer part of the real number $s $ and for any $ N\in \N $, $ C^{N}_b(\R) $ denotes the space of functions $f\in  C^N(\R) $ with $ f$, $ f' $, .., $f ^{(N)} $  bounded.

We first introduce our notion of  solutions to \eqref{KdV1} and \eqref{KdV2}.
\begin{definition} Assume that $ \alpha \in L^\infty_T C^3_b $, $ \beta \in L^\infty_T C^2_b$, $\gamma, \epsilon \in L^\infty_T C^1_b$ 
 and $ \delta\in L^\infty(]0,T[\times \R) $. 
 
We say that $ u\in L^\infty_T L^2_x $ is a weak solution to \eqref{KdV1} if for any $ \phi\in C_c^\infty (]-T,T[\times\R) $ it holds 
\begin{align} 
\int_0^T \int_{\R} u \Bigl[ -\phi_t -\partial_x^3 (\alpha \phi) &+\partial_x^2 (\beta \phi) -\partial_x(\gamma \phi) +\delta \phi \Bigr] 
 \, dx \, dt + \frac{1}{2} \int_0^T \int_{\R} u^2 \partial_x(\epsilon \phi) \, dx \, dt \nonumber \\
  & +\int_{\R} u_0(x) \phi(0,x) \, dx =0 \label{weak1}
\end{align}
\end{definition}
\begin{remark}\label{rem1}
Note that if $ u \in L^\infty_T L^2_x $ is a weak solution to  \eqref{KdV1} then \eqref{KdV1} is satisfied in the distributional sense on 
 $]0,T[\times \R $ and thus $ u_t\in L^\infty_T H^{-3}_x $. This forces  $u$ to belong to $ C_{w}([0,T]; L^2(\R)) $ and 
  \eqref{weak1}  ensures that  $u(0)=u_0 $.
\end{remark}
We define in the same way the weak solutions to \eqref{KdV2}. 
\begin{definition}
Assume that $ b\in  L^\infty_T C^2_b $, $ c,e \in  L^\infty_T C^1_b$ 
 and $ d, f \in  L^\infty(]0,T[\times \R) $. 
 
We say that $ u\in L^\infty_T L^2_x $ is a weak solution to \eqref{KdV2} if for any $ \phi\in C_c^\infty (]-T,T[\times\R) $ it holds 
\begin{align}
\int_0^T \int_{\R} u \Bigl[ -\phi_t -\phi_{3x}  &-\partial_x^2 (b \phi) -\partial_x(c  \phi) +d \phi \Bigr] \, dx \, dt+
\int_0^T \int_{\R} u^2 \Bigl[ \frac{1}{2} \partial_x(e \phi)+f \Bigr] \, dx \, dt  \nonumber \\
  & +\int_{\R} u_0(x) \phi(0,x) \, dx =0 \label{weak2}
\end{align}
\end{definition}
Let us now state our first  result.
 \begin{theorem}\label{th2}
Let $ 	s>\frac{1}{2}$ and $ T\in ]0,+\infty]  $.  Assume  that 
$ b, c, e$  in \\
$ L^\infty(]0,T[; C_b^{[s]+2}(\R)) $  with $  e_t $  in $ L^\infty(]0,T[\times \R)$ and $ d,f\in L^\infty(]0,T[; C_b^{[s]+1}(\R))$. Assume  moreover that 
\begin{equation}
b \ge 0 \quad \text{on }[0,T]\times \R \;   .
\end{equation}
      	Then for all $u_0 \in H^{s}(\R)$, there exist a time $0<T_0=T_0(\|u_0\|_{H^{\frac{1}{2}+}})\le T$ and a  solution 
	$u$ to (\ref{KdV2})
	%bounded 
	in $C([0,T_0];H^{s})\cap  L^2_{[b]}(0,T_0; H^{s+1})$. This solution is the unique weak solution of \eqref{KdV2} that belongs respectively to $L^\infty(0,T_0; H^s)\cap  L^2_{[b]}(0,T_0; H^{s+1})$ and $L^\infty(0,T_0; H^s)$ in  respectively  the cases  $ b\not\equiv 0 $ and $ b\equiv 0 $. Moreover, for any $ R>0 $ the solution-map 
	 $ u_0 \mapsto u $ is continuous from the ball of $ H^s(\R) $ centered at the origin with radius $ R $ into $C([0,T_0(R)];H^{s})$. 
\end{theorem}
\begin{remark}
$L^2_{[b]}(0,T_0; H^{s+1}) $ is defined in Subsection \ref{sect22}.
\end{remark} 
\begin{remark}
The hypotheses on the coefficients $b,c,d,e$ and $ f$ given in the above statement are not optimal.
More accurate hypotheses on the coefficients $b,c,d,e$ and $ f$  involving norms in Zygmund spaces can be found in Remark \ref{rem41}.
\end{remark} 
By a suitable change of unknown we will be able to link the solutions of \eqref{KdV1} to the ones of \eqref{KdV2}. As a consequence of  the above theorem we then get the following result for  \eqref{KdV1}. 
\begin{theorem}\label{th1}
  Let $ 	s>\frac{1}{2}$, $ T\in ]0,+\infty]  $  and assume  that 
  $\alpha\in L^\infty(]0,T[; C_b^{[s]+4}(\R)) $ with $\alpha_t\in  L^\infty(]0,T[; C_b^{[s]+1}(\R)) $  
$ \beta, \gamma, \epsilon$  in $ L^\infty(]0,T[; C_b^{[s]+2}(\R)) $  with $  \epsilon_t $  in $ L^\infty(]0,T[\times \R)$ and
 $ \delta\in L^\infty(]0,T[; C_b^{[s]+1}(\R))$. Assume moreover that 
	\begin{itemize}
	\item[1.]  There exists $\alpha_0>0 $ such that for all $ (t,x)\in [0,T] \times \R $, 
$$
\alpha_0 \le \alpha(t,x) \le \alpha_0^{-1} \; .
$$
\item[2.] $$\sup_{(t,x)\in [0,T]\times \R}  \Bigl| \int_0^x (\alpha^{-4/3} \alpha_t)(t,y)dy\Bigr| <\infty \; . 
$$
\item[3.] $ \beta $ can be decomposed as $ \beta=\beta_1+\beta_2 $ with  $\beta_2\le 0 $, $\beta_1, \beta_2 \in L^\infty(]0,T[;C_b^{[s]+2}) $ 
 such that 
$$
(t,x) \mapsto \int_0^x (\alpha^{-1}\beta_1)(t,y) \, dy \in W^{1,\infty}([0,T]; L^\infty(\R)) \; .
$$
	  \end{itemize}
	  We set $g(t,x)=-\beta_2(t,x) \alpha^{1/3}(t,A(x)) $
       Then for all $u_0 \in H^{s}(\R)$, there exist a time $0<T_0=T_0(\|u_0\|_{H^{\frac{1}{2}+}})\le T $ and a  solution 
	$u$ to \eqref{KdV1}
	%bounded 
	in $C([0,T_0];H^{s})\cap  L^2_{[g]}(0,T_0; H^{s+1})$. This solution is the unique weak solution of \eqref{KdV1} that belongs  to $L^\infty(0,T_0; H^s)\cap  L^2_{[g]}(0,T_0; H^{s+1})$.  For any $ R>0 $ the solution-map 
	 $ u_0 \mapsto u $ is continuous from the ball of $ H^s(\R) $ centered at the origin with radius $ R $ into $C([0,T_0(R)];H^{s})$. 
\end{theorem}
\begin{remark}
It is worth noticing that point 3. of the above theorem is satisfies if  there exists $ R>0 $ such that 
$$
\beta\le 0 \quad \text{on} \quad [0,T_0] \times (\R\setminus[-R,R]) \; .
$$
Indeed, we can then decompose $ \beta $ as $ \beta =\beta_1+\beta_2 $ with 
$\beta_1\equiv 0 $ on $\R\setminus[-R_0,R_0[$ with $R_0>R,$ that clearly satisfies point 3. This means that, when the anti-dissipation is confined  in a fixed compact set  for all $t\in [0,T] $, the Cauchy problem associated to \eqref{KdV1}  is locally well-posed in the Hadamard sense in $ H^s $.
\end{remark}

\begin{remark}
If Hypothesis 3. in Theorem \ref{th1} holds with $ \beta_1=\beta $ (i.e. $\beta_2=0$) then  the change of unknown  does link the solution to \eqref{KdV1} to  a solution of \eqref{KdV2} with $ b\equiv 0 $ on $\R $. Therefore, on account of Theorem \ref{th2}, we obtain that in this case \eqref{KdV1}  is actually unconditionally locally well-posed in $ H^s(\R) $.\end{remark}

The rest of this  paper is organized as follows. In the next section we introduce some notations, define our resolution spaces and recall some technical lemmas that will be used in Section  \ref{sect4} to prove estimates on solutions to (\ref{th2}). Note that the proof of some of these lemmas are postponed to the appendix. In Section  \ref{sect3} we establish  the links between the problems \eqref{KdV1} and \eqref{KdV2} that enables us to  prove Theorem \ref{th1} assuming Theorem \ref{th2}. Finally,    Sections \ref{sect4} and  \ref{sect5} are devoted to the proof of Theorem \ref{th2}. 
  \section{Notations, function spaces and technical lemmas} \label{sect2}
\subsection{Notations}\label{sect21}
For any $s \in \R$, we denote $[s]$ the integer part of $s$.   For $\alpha \in \mathbb R$, $\alpha_+$, respectively
$\alpha_-$, will denote a number slightly greater, respectively  lesser,  than $\alpha$.

 % Let $C_0$ denotes a   {\color{red}  positive } constant whose exact expression is of no %importance. {\color{red}  For $(a,b)\in (\R_+)^2$,  $a\lesssim b$ means that 
% $a\leq C_0\ b$, $ a \gtrsim b $ means that $a\ge C_0 b $  and $ a\sim b $ means that $ %C_0^{-1} b \le a \le C_0 b $. 
  For $(a,b)\in (\R_+)^2$,  We  denote  by respectively $ a\vee b$ and $ a\wedge b$ the maximum and the minimum  of $ a$ and $ b $.  \\
 We denote by $C(\lambda_1, \lambda_2,\dots)$ a nonnegative constant depending on the parameters
 $\lambda_1$, $\lambda_2$,\dots and whose dependence on the $\lambda_j$ is always assumed to be nondecreasing.\\
%We use the condensed notation
%\[
%A_s=B_s+\left\langle C_s\right\rangle_{s>\underline{s}},
%\]
%to express that $A_s=B_s$ if $s\leq \underline{s}$ and $A_s=B_s+C_s$ if $s> \underline{s}$.\\
 Let $p$ be any constant
 with $1\leq p< \infty$ and denote $L^p=L^p(\R)$ the  space of all Lebesgue-measurable functions
 $f$ with the standard norm $$ \|f \|_{L^p}=\big(\int_{\R}\vert f(x)\vert^p dx\big)^{1/p}<\infty.$$
  The real inner product of any two functions $f_1$
 and $f_2$ in the Hilbert space $L^2(\R)$ is denoted by
\[
 (f_1,f_2)=\int_{\R}f_1(x)f_2(x) dx.
 \]
 The space $L^\infty=L^\infty(\R)$ consists of all essentially bounded and Lebesgue-measurable functions
 $f$ with the norm
\[
 \| f\|_{L^\infty}= \sup \vert f(x)\vert<\infty.
\]
 We denote by $W^{1,\infty}(\R)=\{f\in {\mathcal D}'(\R), \mbox{ s.t. }f,\partial_x f\in L^{\infty}(\R)\}$ endowed with its canonical norm.\\
 For any real constant $s\geq0$, $H^s=H^s(\R)$ denotes the Sobolev space of all tempered
 distributions $f$ with the norm $\| f\|_{H^s}=\| \Lambda^s f\|_{L^2} < \infty$, where $\Lambda$
 is the pseudo-differential operator $\Lambda=(1-\partial_x^2)^{1/2}$.\\
%For a given $\mu>0$, we denote by $H^{s+1}_\mu(\R)$ the space $H^{s+1}(\R)$ endowed with the norm
%\[\big\vert\ \cdot\ \big\vert_{H^{s+1}_\mu}^2 \ \equiv \ \big\vert\ \cdot\ \big\vert_{H^{s}}^2\ + \ \mu \big\vert\ \cdot\ \big\vert_{H^{s+1}}^2\ .\]
 For any two functions $u=u(t,x)$ and $v(t,x)$
 defined on $ [0,T)\times \R$ with $T>0$, we denote the $ H^s$  inner product, the $L^p$-norm and especially
 the $L^2$-norm, as well as the Sobolev norm,
 with respect to the spatial variable $x$, by $(u,v)=(u(t,\cdot),v(t,\cdot))_{H^s}$, $\|u \|_{L^p}=\| u(t,\cdot)\|_{L^p}$,
 $\| u \|_{L^2}=\| u(t,\cdot)\|_{L^2}$ , and $ \| u \|_{H^s}=\| u(t,\cdot)\|_{H^s}$, respectively.\\
 We denote $L^\infty([0,T);H^s(\R))$ the space of functions such that $u(t,\cdot)$ is controlled in $H^s$, uniformly for $t\in[0,T)$: $\big\Vert u\big\Vert_{L^\infty([0,T);H^s(\R))} \ = \ \sup_{t\in[0,T)}\vert u(t,\cdot)\vert_{H^s} \ < \ \infty.$\\
 Finally, $C^k(\R)$ denotes the space of
 $k$-times continuously differentiable functions.\\
 
Throughout the paper, we fix a smooth even bump function $\eta$ such that
\begin{equation}\label{defeta}
\eta \in C_0^{\infty}(\mathbb R), \quad 0 \le \eta \le 1, \quad
\eta_{|_{[-1,1]}}=1 \quad \mbox{and} \quad  \mbox{supp}(\eta)
\subset [-2,2].
\end{equation}
We set  $ \phi(\xi):=\eta(\xi)-\eta(2\xi) $. For $l \in \mathbb \N\setminus\{0\}$, we define
\begin{displaymath}
\phi_{2^l}(\xi):=\phi(2^{-l}\xi) \quad \text{and}\quad 
\psi_{2^{l}}(\xi,\tau)=\phi_{2^{l}}(\tau-\xi^3)\; .
\end{displaymath}
By convention, we also denote
\begin{displaymath}
\phi_{1}(\xi):=\eta(\xi) \quad \text{and}\quad \psi_{{1}}(\xi,\tau):=\eta(\tau-\xi^3).
\end{displaymath}
Any summations over capitalized variables such as $N, \, L$, $K$ or
$M$ are presumed to be dyadic. {Unless stated otherwise, we work with non-homogeneous decompositions for space, time and  modulation variables, i.e. these variables range over numbers of the form  $\{2^k : k\in\mathbb N\}$ respectively.}  Then, we have that
{
\begin{displaymath}
\sum_{N\ge 1}\phi_N(\xi)=1\quad \forall \xi\in \R, \quad \mbox{supp} \, (\phi_N) \subset
\{\frac{N}{2}\le |\xi| \le 2N\}, \ N \in \{2^k : k\in \N\setminus\{0\}\},
\end{displaymath}
and
\begin{displaymath}
\sum_{L\ge 1}\psi_L(\xi,\tau)=1 \quad \forall (\xi,\tau)\in\R^2, \quad  L \in \{2^k : k\in \mathbb N\}.
\end{displaymath}
}
Let us now define the following Littlewood-Paley multipliers : 
\begin{equation}
P_Nu=\mathcal{F}^{-1}_x\big(\phi_N\mathcal{F}_xu\big), \quad
Q_Lu=\mathcal{F}^{-1}\big(\psi_L\mathcal{F}u\big), \quad R_K u =\mathcal{F}^{-1}_t\big(\phi_K\mathcal{F}_t u\big)\;.
\label{project}
\end{equation}
We then set 
 $$
\tilde{P}_N:= \hspace*{-5mm}\sum_{ N/4 \le K \le 4N} \hspace*{-5mm} P_K, \quad
 P_{\ge N}:=\sum_{K \ge N} P_{K},  \quad P_{\le N}:=\sum_{1\le K \le N} P_{K},
  \quad P_{\ll N} :=\sum_{1\le K \ll N} P_{K}, $$
 $$
   P_{\gtrsim  N}:=\sum_{ K \gtrsim  N} P_{K}, \quad Q_{\ge L}:=\sum_{K \ge L} Q_{K}, \quad   
 Q_{\le L}:=\sum_{1\le K \le L} Q_{K} \;\text{ and }\;  Q_{\sim L}:=\sum_{ K \sim N} Q_{K}\; .
 $$
 
  For brevity we also write $u_N=P_Nu$, $u_{\le N}=P_{\le N}u$,  $ u_{\ge N} =P_{\ge N}u$,
   $ u_{\ll N}  =P_{\ll N} u $ and  $ u_{\gtrsim N}= P_{\gtrsim N} u $.
 
 Following \cite{david}, to handle coefficient that are not asymptotically flat we will use  the classical Zygmund spaces : for $ s\in \R $, $C_\ast^s(\R)  $ is the set of all $v\in {\mathcal S}'(\R) $ 
such that 
\begin{equation}\label{defZyg}
\|v\|_{C_\ast^s}:=\sup_{N\ge 1} N^s \|P_N v\|_{L^\infty}<\infty \; .
\end{equation}
Note that, for all $k\in \N $, 
$$
C^{k+}_\ast(\R)\hookrightarrow W^{k,\infty}(\R) \hookrightarrow C^k_\ast(\R)\; .
$$
\subsection{Function Spaces}\label{sect22}
Let $T>0$,  $ b\in L^\infty(]0,T[\times \R) $ with $ b\ge 0 $ and $ \theta>-1/2$. We define the sub vector space 
 $ L^2_{[b]}(]0,T[;H^{\theta+1}) $ of  $ L^\infty(0,T;L^2(\R)) $ as 
 $$
  L^2_{[b]}(]0,T[;H^{\theta+1})=\Bigl\{u\in L^\infty(0,T;L^2(\R)), \quad \|u\|_{L^2_{[b]}(]0,T[;H^{\theta+1})}<+\infty \|  \Bigr\}
 $$
 with 
 \begin{equation}\label{defL2b}
\|u\|_{L^2_{[b]}(]0,T[;H^{\theta+1})}^2=\sum_{N>0}\cro{N}^{2\theta} \Bigl\|\sqrt{b} \, P_N u_x \Bigr\|_{L^2_T L^2_x}^2
\end{equation}
For $s$, $\theta \in \mathbb R$, we introduce the Bourgain spaces
$X^{s,\theta}$ related to the linear KdV equation as
the completion of the Schwartz space $\mathcal{S}(\mathbb R^2)$
under the norm
\begin{equation} \label{Bourgain}
\|v\|_{X^{s,\theta}} := \left(
\int_{\mathbb{R}^2}\langle\tau-\xi^3\rangle^{2\theta}\langle \xi\rangle^{2s}|\widehat{v}(\xi, \tau)|^2
d\xi d\tau \right)^{\frac12},
\end{equation}
where $\langle x\rangle:=1+|x|$.
 Recall that $$ \|v\|_{X^{s,\theta}}= \| U(-t) v \|_{H^{s,\theta}_{x,t}} $$ where $ U(t)=\exp(-t \partial_x^3) $ is the generator of the free evolution associated with
  the linear KdV equation and  where $\| \cdot \|_{H^{s,\theta}_{x,t}}$ is the usual space-time Sobolev norm  
 given by 
  $$
\|u\|_{H^{s,\theta}_{x,t}}:= \left(
\int_{\mathbb{R}^2}\langle\tau\rangle^{2\theta}\langle \xi\rangle^{2s}|\widehat{u}(\xi, \tau)|^2
d\xi d\tau \right)^{\frac12}\; .
  $$
  
  We define the function space $ Y^s $ by   $ Y^s =L^\infty_t H^s_x
  \cap X^{s-1,1}$ equipped with its natural norm
  \begin{equation}\label{defZs}
  \|u\|_{Y^s}= \|u\|_{L^\infty_t H^s_x} +\|u\|_{X^{s-1,1}} \; .
  \end{equation}
Finally, we will use restriction in time versions of these spaces.
Let $T>0$ be a positive time and $Y$ be a normed space of space-time functions. The restriction space $ Y_T $ will
 be the space of functions $v: \mathbb R \times
]0,T[\rightarrow \mathbb R$ satisfying
\begin{displaymath}
\|v\|_{Y_{T}}:=\inf \{\|\tilde{v}\|_{Y} \ | \ \tilde{v}: \mathbb R
\times \mathbb R \rightarrow \mathbb R, \ \tilde{v}|_{\mathbb R
\times ]0,T[} = v\}<\infty \; .
\end{displaymath}

 \subsection{Technical Lemmas}
 We first recall the following technical lemmas that were proven in \cite{MV}.
\begin{lemma}\label{continuiteQ}
Let $ { L\ge 1 } $, $ 1\le p \le \infty $ and $ s\in \R $.
The operator  $ Q_{\le L} $ is bounded  in $ L^p_t H^s $ uniformly in ${ L\ge 1 }$.
\end{lemma}
For any $T>0$, we consider $1_T$ the characteristic function of $[0,T]$ and use the decomposition
\begin{equation}\label{ind-dec}
1_T = 1_{T,R}^{low}+1_{T,R}^{high},\quad \widehat{1_{T,R}^{low}}(\tau)=\eta(\tau/R)\widehat{1_T}(\tau)
\end{equation}
for some $R>0$.
\begin{lemma}\label{ihigh-lem} For any $ R>0 $ and $ T>0 $ it holds
\begin{equation}\label{high}
\|1_{T,R}^{high}\|_{L^1}\lesssim T\wedge R^{-1}.
\end{equation}
and, for any $ p\in [1,+\infty] $,
\begin{equation}\label{low}
\|1_{T,R}^{low}\|_{L^p}+ \|1_{T,R}^{high}\|_{L^p} \lesssim  T^{1/p} 
\end{equation}
\end{lemma}
\begin{lemma}\label{ilow-lem}
Let $u\in L^2(\R^2)$. Then for any $T>0$, $R>0$  and $ L \gg R $ it holds
$$
\|Q_L (1_{T,R}^{low}u)\|_{L^2}\lesssim \| Q_{\sim L} u\|_{L^2}
$$
\end{lemma}
We will need product estimates in Sobolev spaces for functions in Sobolev and in Zygmund spaces (see \cite{AG} for \eqref{estsobo} and \cite{david} for \eqref{estCs}. The proof of  \eqref{estC} follows exactly the same lines as the one of  \eqref{estsobo}).
\begin{lemma}\label{product}
 1. Let $(t,s,r)\in \R^3 $ with $ s+r>t+1/2$, $s+r>0 $ and $s,r\ge t $. Then 
 for any $ f\in H^s(\R) $ and $ g\in H^r(\R) $, it holds $ fg \in H^t(\R) $ with 
 \begin{equation}\label{estsobo}
 \|fg\|_{H^t} \lesssim \|f\|_{H^s} \|g\|_{H^r}
 \end{equation}
2.  Let $(t,s,r)\in \R^3 $ with $ s+r>t$, $s+r>0 $ and $s,r\ge t $. Then 
 for any $ f\in C^s_*(\R) $ and $ g\in H^r(\R) $, it holds $ fg \in H^t(\R) $ with 
 \begin{equation}\label{estC}
 \|fg\|_{H^t} \lesssim \|f \|_{C^s_*} \|g\|_{H^r}
 \end{equation}
In particular, let $ s\in\R $, then 
 for any $ f\in C^{|s|+}_*(\R) $ and $ g\in H^s(\R) $, it holds $ fg \in H^s(\R) $ with 
 \begin{equation}\label{estCs}
 \|fg\|_{H^s} \lesssim \|f \|_{C^{|s|+}_*} \|g\|_{H^s}
 \end{equation}

 \end{lemma}
We will also need the following lemma on commutator and double commutator estimates (see ( [\cite{david}, p. 288] the remark in the footnote for \eqref{commu})  that we prove  in the Appendix.
\begin{lemma}\label{commutator}
 Let $ f\in L^\infty(\R) $ and $ g\in L^2(\R) $. For any $ N>0  $ it holds 
 \begin{equation}\label{commu}
  \|[P_{N}, P_{\ll N}f ] g \|_{L^2_x} \lesssim N^{-1}\|P_{\ll N}f_x \|_{L^\infty_x} \| {\color{blue}\tilde{P}_N} g \|_{L^2_x} 
 \end{equation}
and 
 \begin{equation}\label{commu2}
  \Bigl\| \Bigl[P_N,[P_{N}, P_{\ll N}f ]\Bigr] g \Bigr\|_{L^2_x} \lesssim  N^{-2}\|P_{\ll N}f_{xx} \|_{L^\infty_x} \| \tilde{P}_N g \|_{L^2_x} 
 \end{equation}
 Moreover, it holds 
 \begin{equation}\label{comcom}
  \int_{\R} [P_N,P_{\ll N}f] g \; P_N g = \frac{1}{2} \int_{\R} \Bigl[P_N,[P_{N}, P_{\ll N}f ]\Bigr] \tilde{P}_N g\, \tilde{P}_N g 
\end{equation}
 \end{lemma}
 Finally we construct a bounded linear operator from  $ X^{s-1,1}_T\cap L^\infty_T H^s_x $
 into $ Y^s $ with a bound  that does not depend on $ s $ and $ T $ .  For this we follow \cite{MN} and introduce the extension operator $ \rho_T $ defined by
\begin{equation}\label{defrho}
\rho_T (u)(t):= U(t)\eta(t) U(-\mu_T(t)) u(\mu_T(t))\; ,
\end{equation}
where $ \eta $ is the smooth cut-off function defined in Section \ref{sect21} and $\mu_T $ is the  continuous piecewise affine
 function defined  by
\begin{equation}\label{defext}
 \mu_T(t)=\left\{\begin{array}{rcl}
 0&\text{for}&t\not\in ]0,2T[ \\
 t&\text {for}&t\in [0,T] \\
 2T-t&\text {for}&t\in [T,2T]
 \end{array}
 \right.
\end{equation}
\begin{lemma} \label{extension}
Let $0<T \le 2$ and $s \in \mathbb R$. Then,
\begin{displaymath}
\begin{split}
\rho_T : \ & X^{s-1,1}_T\cap L^\infty_T H^s_x \longrightarrow  Y^s\\
 &u \mapsto \rho_T(u)
\end{split}
\end{displaymath}
is a bounded linear operator, \textit{i.e.}
\begin{equation} \label{extension.1}
\|\rho_T(u)\|_{L^{\infty}_t H^s_x}  + \|\rho_T(u)\|_{X^{s-1,1}} 
\lesssim  \|u\|_{L^\infty_T H^s_x}+\|u\|_{X^{s-1,1}_T} \, ,
\end{equation}
for all $u \in X^{s-1}_T\cap L^\infty_T H^s_x$.

Moreover, the implicit constant in \eqref{extension.1} can be chosen independent of $0<T \le {\color{red}2}$ and $s\in \R $.
\end{lemma}
\section{Transformation of the problem and proof of Theorem \ref{th1}.}\label{sect3}
\subsection{Link between solutions of \eqref{KdV1} and \eqref{KdV2}}
The main assumption on the coefficient of the third order term is that it  is bounded from above and from below by positive constants. Of course, we can also treat the case of a negative coefficient by making the trivial change of unknwon $ \tilde{u}(t,x)=u(t,-x) $ but this will also change the sens of the real axis. This would play no role in Theorem \ref{th1}  but would change the assumption $\sup_{(t,x)\in [0,T]\times \R}  - \int_0^x \frac{\beta_1}{\alpha}(t,y)dy <\infty $ by 
 $\sup_{(t,x)\in [0,T]\times \R}   \int_0^x \frac{\beta_1}{\alpha}(t,y)dy <\infty $ in Theorem \ref{th3} below. 
\begin{hypothesis} \label{hyp1} There exists $\alpha_0>0 $ such that for all $ (t,x)\in [0,T] \times \R $, 
$$
\alpha_0 \le  \alpha(t,x) \le \alpha_0^{-1} \; .
$$
\end{hypothesis}
\begin{proposition}\label{prop31}
Assume that Hypothesis \ref{hyp1} is satisfied and that 
$\alpha\in L^\infty(]0,T[; C_b^3(\R)) $ with $\alpha_t \in  L^\infty(]0,T[; C_b(\R)) $ and $  \beta\in  L^\infty(]0,T[; C_b^2(\R)) $.
 Let $ A\in L^\infty(]0,T[; C_b^4(\R))  $ with $A_t \in  L^\infty(]0,T[; C_b^1(\R)) $  be defined for $ (t,x)\in [0,T] \times \R $ by 
\begin{equation}\label{defA}
A(t,x)=\int_0^x \alpha^{-1/3}(t,y) \, dy 
\end{equation}
and let $ h>0  $  with $ h\in L^\infty(]0,T[; C_b^3(\R))  $ with $h_t \in  L^\infty(]0,T[;  C_b(\R)) $. For each $ t\in [0,T] $ we denote by $A^{-1}(t,\cdot) $
 the  increasing reciproqual bijection of $ A(t,\cdot) $. 
 
Then $ u\in L^\infty_T L^2_x $ is a weak solution to \eqref{KdV1} if and only if 
$$ (t,x) \mapsto v(t,x)=
h(t,A^{-1}(t,x))\,  u(t,A^{-1}(t,x)) $$ is a weak solution to \eqref{KdV2} with 
\begin{equation}\label{defb}
\left\{ 
\begin{array}{ll}
b(t,x)& =\alpha^{1/3} \Bigl(-\beta \alpha^{-1}+\alpha_x \alpha^{-1}+3 h^{-1} h_x\Bigr) \\
c(t,x) &=A_t +\alpha^{-1/3} \Bigl( 6 h_x^2 h^{-2} \alpha + \frac{4}{9} \alpha_x^2 \alpha^{-1} + \alpha_x h_x h^{-1}
-3 h_{2x} h^{-1} \alpha - \frac{1}{3} \alpha_{2x}\\ % \alpha^{-1} \\
 & \quad -2 h_x h^{-1} \beta -\frac{1}{3} \alpha^{-1} \alpha_x \beta +\gamma\Bigr)
  \\
d(t,x)& =\alpha \Bigl(-6 h_x^3 h^{-3} +6 h_{2x} h^{-2} h_x -h_{3x} h^{-1}\Bigr)+ \beta  \Bigl( 2 h_x^2 h^{-2} -  h_{2x} h^{-1}\Bigr) \\
& \quad - \gamma h_x h^{-1} -h_t h^{-1}+\delta \\
e(t,x)& =\epsilon \alpha^{-1/3} h^{-1} \quad \text{ and  } f(t,x)= - \epsilon h_x h^{-2} \; .
\end{array}
\right . 
\end{equation}
where all the functions in the right-hand side are evaluated at $ (t,A^{-1}(t,x)) $.

\end{proposition}
\begin{proof}

Since $\alpha\ge \alpha_0>0 $ on $[0,T] \times \R $, for each $t\in [0,T] $, $ A(t,\cdot) $ is an increasing bijection of $ \R $ with no critical point and thus its reciprocal bijection $A^{-1}(t,\cdot) $ is well-defined and belong to the same $ C^n $-space. Therefore, since $\alpha\in L^\infty(]0,T[; C_b^3(\R)) $ with $\alpha_t \in  L^\infty(]0,T[; C_b(\R)) $, it is clear that $A$ and $ A^{-1}$ belong to $  L^\infty(]0,T[; C^4_b(\R)) \cap W^{1,\infty}([0,T]; C^1_b(\R)) $

We first assume that $ u \in C([0,T]; H^\infty) $ with $ u_t \in L^\infty(]0,T[; H^\infty) $ and we set 
\begin{equation}
\label{defV}
 V(t,x)=h(t,A^{-1}(t,x)) \; u(t,A^{-1}(t,x))  
 \end{equation}
 so that 
 $$
  u(t,x)= \frac{V(t,A(t,x))}{h(t,x)} \; 
  $$
 
 In the calculus below   the functions $ u, \, h\,, \alpha,\, \beta, \, \gamma, \, \delta $ $ \epsilon $ will be evaluated  at $(t,x)$
  whereas $ V $ is evaluated at $(t,A(t,x))$. Then it holds 
\begin{align*}
u_t(t,x)= &  -h_t h^{-2} V +h^{-1}  V_t +A_t h^{-1} V_x\\
u_x(t,x)=& -\frac{h_x}{h^2} V+\frac{\alpha^{-1/3}}{h} V_x \\
u_{2x}(t,x)= & \alpha^{-2/3} h^{-1} V_{2x} -\Bigl( \frac{h^{-1}}{3} \alpha^{-4/3} \alpha_x +2 h_x h^{-2} \alpha^{-1/3} \Bigr) V_x
\\ 
&\hspace*{0.5cm}+ \Bigl( 2 h_x^2 h^{-3} -h_{2x} h^{-2} \Bigr) V   \\
u_{3x}(t,x)=& \alpha^{-1} h^{-1} V_{3x} +V_{2x} \Bigl( -h^{-1} \alpha^{-5/3}  \alpha_x-3 h_x h^{-2} \alpha^{-2/3} \Bigr) \\
&\hspace*{0.5cm}+V_x \Bigl( h_x h^{-2} \alpha^{-4/3} \alpha_x+\frac{4}{9} h^{-1} \alpha^{-7/3} \alpha_x^2-\frac{1}{3} h^{-1} \alpha^{-4/3}
\alpha_{2x} \\
&\hspace*{4cm}-3 h_{2x} h^{-2}  \alpha^{-1/3}
+6 h_x^2 h^{-3}  \alpha^{-1/3} \Bigr) \\
& \hspace*{0.5cm}+ V \Bigl( 6 h_{2x} h_x h^{-3}-6 h_x^3 h^{-4} -h_{3x} h^{ -2}\Bigr) \\
(u u_x)(t,x)  = & -h^{-3} h_x V^2 +\alpha^{-1/3} h^{-2} V V_x \; .
\end{align*}
Gathering the above identity we thus obtain 
\begin{align}
h(t,x)& \Bigl(u_t +\alpha u_{3x} + \beta u_{2x} +\gamma u_x + \delta u  - \epsilon u u_x \Bigr)(t,x)\nonumber \\
&=[V_t+V_{3x} -b V_{2x} +c V_x + d V -e V V_x-f V^2](t,A(t,x))
\end{align}
with $ b,c,d,e $ given  by \eqref{defb}.

Therefore for  $ \phi \in L^\infty(]0,T[; C_b^3(\R)) $ with $\phi_t \in  L^\infty(]0,T[; C_b(\R)) $ and compact support in $ [0,T[ \times \R $, making use at any fixed $t \in [0,T] $ of the change of variable $ y=A^{-1}(t,x) $ and noticing that 
 $ A^{-1}_x(t,x)=\alpha^{1/3} (t,A^{-1}(t,x)) $  we observe that 
\small{\begin{align}
\int_0^T \int_{\R} &\Bigl(u_t +\alpha u_{3x} + \beta u_{2x} +\gamma u_x + \delta u  - \epsilon u u_x \Bigr)(t,y) \phi(t,y) \, dy \nonumber \\
=&\int_0^T \int_{\R} h \Bigl(u_t +\alpha u_{3x} + \beta u_{2x} +\gamma u_x + \delta u  - \epsilon u u_x \Bigr)(t,y)
\frac{\phi}{h}(t,y)\, dy  \nonumber \\
= \int_0^T \int_{\R} &\Bigl[ h \Bigl(u_t +\alpha u_{3x} + \beta u_{2x} +\gamma u_x + \delta u  - \epsilon u u_x \Bigr)
\frac{\phi}{h}\Bigr] (t,A^{-1}(t,x)) \; \alpha^{1/3}(t,A^{-1}(t,x))\, dx \, dt \nonumber \\
=&  \int_0^T \int_{\R} \Bigl(V_t+V_{3x} -b V_{2x} +c V_x + d V - e V V_x-f V^2 \Bigr)(t,x)
\psi(t,x) \, dx \, dt \nonumber \\
=& \int_0^T \int_{\R} V \Bigl[ -\psi_t -\psi_{3x}  -\partial_x^2 (b \psi) -\partial_x(c  \psi) +d \psi \Bigr] +V^2 \Bigl[ \frac{1}{2} \partial_x(e \psi)+f \Bigr] \, dx \, dt  \nonumber \\
&\qquad+\int_{\R} V(0,x) \psi(0,x) \, dx \label{weak11}
\end{align}}

with $ \psi(t,x)={\displaystyle \frac{ \alpha^{1/3} \,\phi }{h} (t,A^{-1}(t,x)) }$.\\

Now let $ u\in L^\infty_T L^2_x $ be a weak solution to \eqref{KdV1}.  Recall that by Remark \ref{rem1}, $ u_t\in L^\infty_T H^{-3}_x $. Then by using mollifiers we can approximate 
$ u $ in $L^\infty_T L^2_x$ by $ u_n\in C([0,T]:H^\infty) $ with $ u_t\in L^\infty(]0,T[; H^\infty) $  such that $ u_n(0)\to u_0 $ in $ L^2(\R) $ and 
 $ u_n\to u \in L^\infty_T L^2_x $. Note that by defining $ V_n $ in the same way as $V$  in \eqref{defV} we also have 
 $ V_n(0)\to V_0 $ in $ L^2(\R) $ and 
 $ V_n\to V \in L^\infty_T L^2_x $.
Making use of \eqref{weak11} and that $ u$ is a weak solution   to \eqref{KdV1} we thus get
\begin{align}
&0=  \int_0^T \int_{\R} \Bigl(u 
\Bigl[ -\phi_t -\partial_x^3 (\alpha \phi) +\partial_x^2 (\beta \phi) -\partial_x(\gamma \phi) +\delta \phi \Bigr] 
  + \frac{1}{2}  u^2 \partial_x(\epsilon \phi)\Bigr)(t,x) \, dx \, dt \nonumber \\
  &\hspace*{4cm}+\int_{\R} u_0(x) \phi(0,x) \, dx  \nonumber \\
&=\lim_{n\to +\infty} \int_0^T \int_{\R} \Bigl(u_n 
\Bigl[ -\phi_t -\partial_x^3 (\alpha \phi) +\partial_x^2 (\beta \phi) -\partial_x(\gamma \phi) +\delta \phi \Bigr] 
  + \frac{1}{2}  u^2_n \partial_x(\epsilon \phi)\Bigr)(t,x) \, dx \, dt \nonumber \\
  &\hspace*{4cm} +\int_{\R} u_{n}(0,x) \phi(0,x) \, dx   \nonumber \\
  & =\lim_{n\to +\infty} \int_0^T \int_{\R}   \Bigl(u_{n,t} +\alpha u_{n,3x} + \beta u_{n,2x} +\gamma u_{n,x} + \delta u_n  - \epsilon u_n u_{n,x} \Bigr)(t,x) \phi(t,x) \, dx\, dt \nonumber \\
  & =\lim_{n\to +\infty}\int_0^T \int_{\R} V_n \Bigl[ -\psi_t -\psi_{3x}  -\partial_x^2 (b \psi) -\partial_x(c  \psi) +d \psi \Bigr] +V_n^2 \Bigl[ \frac{1}{2} \partial_x(e \psi)+f \Bigr] \, dx \, dt  \nonumber\\
 &\hspace*{4cm} +\int_{\R} V_n(0,x) \psi(0,x) \, dx\nonumber \\
   & = \int_0^T \int_{\R} V \Bigl[ -\psi_t -\psi_{3x}  -\partial_x^2 (b \psi) -\partial_x(c  \psi) +d \psi \Bigr] +V^2 \Bigl[ \frac{1}{2} \partial_x(e \psi)+f \Bigr] \, dx \, dt \nonumber \\
  &\hspace*{4cm}  +\int_{\R} V(0,x) \psi(0,x) \, dx 
\end{align}
that proves that $ u $ is a weak solution to \eqref{KdV1} if and only if : \\ 
$(t,x) \mapsto V(t,x)=h(t,A^{-1}(t,x)) u(t,A^{-1}(t,x)) $ is a weak solution to \eqref{KdV2}. Indeed since $ \alpha, h \in L^\infty(]0,T[; C_b^3(\R)) $, $\alpha_t , h_t \in  L^\infty(]0,T[; C_b(\R)) $ with $ h>0 $ and  $ \alpha\ge \alpha_0>0 $, the map
$$
\Theta \quad :  \quad \phi\mapsto \Bigl(\frac{\phi \, \alpha^{1/3}}{h}\Bigr) (t,A^{-1}(t,x)) 
$$
is a bijection from the space of functions in $ L^\infty(]0,T[; C_b^3(\R)) $ with time derivative in $L^\infty(]0,T[; C_b(\R)) $ and  compact support in 
$[0,T[ \times \R $ into itself. The reciprocal bijection is given by 
$$
\Theta^{-1} \quad  : \quad \psi\mapsto \Bigl(\frac{\psi h}{\alpha^{1/3}}\Bigr) (t,A(t,x))\; .
$$
\eqref{weak1}  is thus satisfied by all $ \psi\in L^\infty(]0,T[; C_b^3(\R))  $ with  $\psi_t \in L^\infty(]0,T[; C_b(\R)) $ and   compact support in $[0,T[ \times \R $ that leads to the desired result.
\end{proof}
\subsection{Proof of Theorem \ref{th1} assuming Theorem \ref{th2}}
We want to choose $ h $ such that $ b \ge 0 $. For this we decompose $ \beta(\cdot,\cdot) $ as 
$\beta_1+\beta_2 $ with $ \beta_1$ and $ \beta_2 $ bounded  and $ \beta_2\le 0   $ (Note that we can always take $ \beta_1=\beta $
 and $ \beta_2=0 $). According to \eqref{defb} it suffices to take $ h$ that satisfies
\begin{equation}\label{abiir}
\frac{h_x}{h}= \frac{1}{3}(\beta_1 \alpha^{-1}-\alpha_x \alpha^{-1})
\end{equation}
so that
\begin{equation*}
b=-\beta \alpha^{-\frac{2}{3}}+\alpha_x \alpha^{-\frac{2}{3}}+ 3\frac{h_x}{h}\alpha^{1/3}=-\beta_2 \alpha^{-\frac{2}{3}}\ge 0.
\end{equation*}
Equation (\ref{abiir}) is satisfied for 
\begin{equation}\label{defh}
\displaystyle h(t,x)= \Bigl[ \frac{\alpha(t,0)}{\alpha(t,x)}\Bigr]^{1/3}  \exp\Big(\frac{1}{3}{\displaystyle\int_0^x (\beta_1 \alpha^{-1}})(t,y)\,dy\Big)\; .
\end{equation}
For this choice of $ h $ we need  the coefficients $ b,c,d,e,f $ to be bounded to solve the equation with the help of Theorem \ref{th2}. 
First we notice that the coefficient $ c$ contains $ A_t $. The requirement that $ A_t $ is bounded  leads to the following hypothesis.
\begin{hypothesis}\label{hyp2}
$$\sup_{(t,x)\in [0,T]\times \R}  \Bigl| \int_0^x (\alpha^{-4/3} \alpha_t)(t,y)dy\Bigr| <\infty \; . 
$$
\end{hypothesis}
Now, since $ \alpha\ge \alpha_0 $ one can check that all the terms $\frac{h_x}{h} $, $\frac{h_{2x}}{h} $ that appear in $ c$ and $ d $ are bounded. On the other hand the boundedness of $ h_t h^{-1}$ that appears in the coefficient $ d$ requires a new hypothesis. Moreover, in the coefficient $e $ and $f $ of the nonlinear part, $ h^{-1} $ appears alone. To force  $h_t h^{-1} $, $e $ and $ f$ to be bounded we thus  add the following hypothesis that  ensures in particular that there exists $ h_0 >0 $ such that 
 for  $ (t,x)\in [0,T_0]\times \R $, $ h(t,x)\ge h_0 $.
\begin{hypothesis}\label{hyp3}
$ \beta $ can be decomposed as $ \beta=\beta_1+\beta_2 $ with  $\beta_2\le 0 $, $\beta_1,\beta_2 \in L^\infty([0,T];C_b^2) $ 
, $ \partial_t \beta_1\in L^\infty(]0,T[; L^\infty) $  such that 
$$\sup_{(t,x)\in [0,T]\times \R}  \Bigl| \int_0^x \partial_t (\alpha^{-1} \beta_1)(t,y)dy\Bigr| <\infty \; . 
$$
and 
$$\sup_{(t,x)\in [0,T]\times \R}  - \int_0^x \frac{\beta_1}{\alpha}(t,y)dy <\infty \; . 
$$

\end{hypothesis}
Now, according to Theorem \ref{th2}, for $s>1/2$, \eqref{KdV2} is locally well-posed in $ H^s(\R) $, whenever $ b\ge 0 $ on $ [0,T] \times \R $ with 
$ b, c, e$  in $ L^\infty(0,T; C_b^{[s]+2}(\R)) $, $  e_t $  in $ L^\infty(]0,T[\times \R)$ and $ d,f\in L^\infty(]0,T[; C_b^{[s]+1}(\R))$.

In view of \eqref{defb}, \eqref{defh} and Hypotheses \ref{hyp1}-\ref{hyp3}, one can easily check that  the function spaces to which $\alpha, \beta, \gamma,\delta, \epsilon $ and $ \beta_1, \beta_2 $ belong in the statement of Theorem \ref{th1} ensure that $ b, c, e,d $ and $ f$ belong to the above function spaces. Moreover, this ensures that $ u\in C([0,T_0]; H^s) $ if and only if $ V(t,x)=h(t,A^{-1}(t,x)) \; u(t,A^{-1}(t,x))  $ belongs also to this space.
 Therefore, gathering Theorem \ref{th2}  and Proposition \ref{prop31}  leads to the existence of a solution to \eqref{KdV1}  with uniqueness in the space of functions $ u $ such that $ hu \in L^\infty(0,T_0; H^s) $. More precisely, we can state the following slightly less restrictive version of Theorem \ref{th1}.
\begin{theorem}\label{th3}
 Let $ s>1/2$ and  $T\in ]0,+\infty]$ and assume  that 
  $\alpha\in L^\infty(]0,T[; C_b^{[s]+4}(\R)) $ with $\alpha_t\in  L^\infty(]0,T[; C_b^{[s]+1}(\R)) $  
$ \beta, \gamma, \epsilon$  in $ L^\infty(]0,T[; C_b^{[s]+2}(\R)) $  with $  \epsilon_t $  in $ L^\infty(]0,T[\times \R)$ and
 $ \delta\in L^\infty(]0,T[; C_b^{[s]+1}(\R))$. Assume moreover that 	\begin{itemize}
	\item  There exists $\alpha_0>0 $ such that for all $ (t,x)\in [0,T] \times \R $, 
$$
\alpha_0 \le  \alpha(t,x) \le \alpha_0^{-1} \; .
$$
\item $$\sup_{(t,x)\in [0,T]\times \R}  \Bigl| \int_0^x \partial_t (\alpha^{-1/3})(t,y)dy\Bigr| <\infty \; . 
$$
\item $ \beta $ can be decomposed as $ \beta=\beta_1+\beta_2 $ with  $\beta_2\le 0 $, $\beta_1, \beta_2 \in L^{\infty}(]0,T[;C_b^{[s]+2}) $ 
 such that 
$$\sup_{(t,x)\in [0,T]\times \R}  \Bigl| \int_0^x \partial_t (\alpha^{-1} \beta_1)(t,y)dy\Bigr| <\infty \; . 
$$
and 
$$\sup_{(t,x)\in [0,T]\times \R}  -\int_0^x \frac{\beta_1}{\alpha}(t,y)dy <\infty \; . 
$$
 \end{itemize}
	  We set 
	  $$ h(t,x)= \Bigl[ \frac{\alpha(t,0)}{\alpha(t,x)}\Bigr]^{1/3}  \exp\Big(\frac{1}{3}{\displaystyle\int_0^x \beta_1 \alpha^{-1}}\Big) \text{and} \quad 
	  g(t,x)=-\beta_2(t,x) \alpha^{1/3}(t,A(x)) \; .
	  $$
       Then for all $u_0 \in H^{s}(\R)$, there exist a time $0<T_0=T_0(\| u_0\|_{H^{\frac{1}{2}+}})\le T $ and a  solution 
	$u$ to (\ref{weak1})
	%bounded 
	in $C([0,T_0];H^{s})\cap  L^2_{[g]}(0,T_0; H^{s+1})$. This solution is the unique weak solution of \eqref{KdV1} such that $ h u $ belongs  to $L^\infty(0,T_0; H^s)\cap  L^2_{[g]}(0,T_0; H^{s+1})$.  
\end{theorem}

\begin{remark}\label{remark31}
It is worth noticing that we can always choose $ (\beta_1,\beta_2) $ such that 
the hypothesis of integrability on $ \beta_1 \alpha^{-1} $ in the above theorem is satisfied in $+\infty$. Indeed, $ \beta $ being bounded  by hypothesis, taking $\beta_2$ such that $ \beta_2= -\sup_{\R} |\beta|  $ on $ \R_ + $ it follows that 
 $ \beta _1=\beta -\beta_2\ge 0 $ on $ \R_+ $  and thus $ \int_0^x \frac{\beta_1}{\alpha}(t,y)dy \ge 0 $ for any $x\in \R_+ $. 
 That means that this existence and uniqueness result 
 works with  a uniform anti-diffusion in the neighborhood of $ +\infty $. For instance  a coefficient $ \beta $ such that $ \beta \ge 1 $ on $ [0,T] \times \R_+ $. This lost of symmetry between $ +\infty $ and $-\infty $ is linked to the fact that we imposed that $ \alpha >0  $ so that linear waves  solutions of $ u_t +\alpha u_{3x} =0 $ are travelling only to the left. 
\end{remark}

 Finally,  if we want to get  the well-posedness in the Hadamard sense of \eqref{KdV1} we need to require a little more on $ h $ so that $ \|u(t)\|_{H^s} \sim \| (h u)(t) \|_{H^s} $ uniformly on $ [0,T_0] $.
 This forces $ h $ to be situated between two positive values, i.e. there exists $ h_0,h_1>0 $ such that for any $ (t,x)\in [0,T]\times \R $, $h_0\le h(t,x)\le h_1$. 

For this it suffices to replace Hypothesis \ref{hyp3} by the following one :

\begin{hypothesis}\label{hyp4}
$ \beta $ can be decomposed as $ \beta=\beta_1+\beta_2 $ with  $\beta_2\le 0 $, $\beta_1,\beta_2 \in L^\infty([0,T];C_b^2) $ 
, $ \partial_t \beta_1\in L^\infty(]0,T[; L^\infty) $  such that 
$$
(t,x) \mapsto \int_0^x (\alpha^{-1}\beta_1)(t,y) \, dy \in W^{1,\infty}([0,T]; L^\infty(\R)) \; .
$$
\end{hypothesis}
which leads to Theorem \ref{th2}. 

\section{Estimates  on the solutions to \eqref{KdV2}}\label{sect4}
In this section, we prove the needed estimates on solutions to \eqref{KdV2} to get the local well-posedness of  \eqref{KdV2} in $ H^s(\R) $ for $ s>1/2$. For this purpose we use the approach introduced in \cite{MV}  that mix energy's and Bourgain's type estimates. 
\subsection{An estimate using  Bourgain's type spaces }
We start  by proving  the only  estimate where we need Bourgain's type spaces. This estimate will be used to bound the contribution of  the nonlinear KdV term $ e u u_x$ 
 in the energy estimate. 
 First we check that under suitable space projections on the functions, we have a good lower bound on the resonance relation that appears in this contribution.
 
 \begin{lemma}\label{resolem}
Let $L_i \ge 1 $ and $N_i  \ge 1  $ be dyadic numbers and $ u_i\in L^2(\R^2) $ for $ i \in \{1,2,3,4\} $.
If  $ N_1  \ll \min(N_2,N_3,N_4) $ then it holds 
$$
\int_{\R^2} P_{N_4} \Bigl(Q_{L_1} P_{\le N_1} u_1 \, Q_{L_2}P_{N_2} u_2\,  Q_{L_3} P_{N_3} u_3\Bigr) Q_{L_4}P_{N_4} u_4=0
$$
whenever the following relation is not satisfied :
 \begin{equation}\label{resonance3}
L_{max}  \sim N_2 N_3 N_4 \text{ or } (  L_{max} \gg  N_2 N_3 N_4 \;  \text{ and } \;  L_{max}\sim L_{med})
\end{equation}
where $ L_{max}=\displaystyle \max_{i=1,..,4}L_i$ and $ L_{med}= \max(\{L_1,L_2,L_3,L_4\}-\{L_{\max}\})$. 
\end{lemma}
\begin{proof} Applying Plancherel identity, this is a  direct consequence of the condition  $ N_1  \ll \min(N_2,N_3,N_4) $  together with  the  cubic resonance relation  associated with the KdV propagator :
$$
\Omega_3(\xi_1,\xi_2,\xi_3)=\sigma\Bigl(-\sum_{i=1}^3 \tau_i, -\sum_{i=1}^3 \xi_i\Bigr)+\sum_{i=1}^3 \sigma(\tau_i,\xi_i) =
-3 (\xi_2+\xi_3)(\xi_1+\xi_3)(\xi_1+\xi_2) 
$$
where   $ \sigma(\tau,\xi):=\tau-\xi^3$. Note that the conditions on the $N_i$'s ensure that  the above integrals vanish for $ L_{max} \lesssim 1$. 
\end{proof}
Now we can  give our main estimate that uses Bourgain's type spaces.
\begin{lemma}\label{lemtriest} 
Assume $ 0<T<1$, $e \in L^\infty_{Tx} $ with $ e_t\in   L^\infty_{Tx} $ and $ u_i\in L^\infty_T H^{-1/2} \cap X^{-\frac{3}{2},1}_T $, $ i=2,3,4$. 
Let $ N_j\in 2^{\N} $, $ j=1,2,3,4$ with $ N_1 \ll \min(N_2, N_3, N_4) $.  Setting, for all  $0<t<T $, 
\begin{equation}\label{II}
I_t^3=I_t(e,u_2,u_3,u_4)=\int_0^t \int_{\R} P_{N_4} (P_{\le N_1}  e \, P_{N_2} u_2 \partial_x P_{N_3}  u_3) P_{N_4} u_4\; ,
\end{equation}
it holds 
\begin{align} 
 |I_t^3 | & \lesssim (\|e\|_{L^\infty_{Tx}}+\|e_t \|_{L^\infty_{Tx}})\Bigl[  \|P_{N_r} u_r\|_{L^\infty_T L^2_x} 
  \Bigl(\sum_{i=p,q} \|P_{N_i} u_i \|_{L^2_{Tx}}\Bigl)  \Bigr(\sum_{i=p,q} \|P_{N_i} u_i \|_{X^{-1,1}_T}\Bigr)  \nonumber \\
    & + T^\frac{1}{16} N_p^{-\frac{1}{4}} \sum_{i=2}^4   \Bigl(\|P_{N_i}u_i\|_{X^{-1,1}_T}
    +\|P_{N_i} u_i \|_{L^\infty_T L^2_x} \Bigr) \prod_{j=2\atop j\neq i}^4\|P_{N_j} u_j \|_{L^\infty_T L^2_x}  
 \Bigr]\label{f1}
\end{align}
whenever $ N_p\sim N_q \gtrsim N_r$ where $(p,q,r) $ is a permutation of $ (2,3,4) $.
\end{lemma}
\begin{proof}
We start by noticing that we may also assume that $ e$ and $ e_t $ belong to $ L^2_T L^2_x  $. 
Indeed, approximating $ e $ by $ e_R=e \, \eta_R $ with $ \eta_R=\eta(\cdot/R) $ where $ \eta $ is the smooth non negative compactly supported function defined in \eqref{defeta}, we notice that  for any $ t\in [0,T] $, Lebesgue dominated convergence theorem leads 
 for any  $ N\in 2^{\N} $ to
$$
{\mathcal F}_x^{-1}(\phi_{\le N}) \ast e_R \to {\mathcal F}_x^{-1}(\phi_{\le N})\ast e = P_{\le N} e\quad  \text{  on }  \R ,
$$
 since 
 $ {\mathcal F}_x^{-1}(\phi_{\le N}) \in L^1(\R) $ and $ |e(t)\, \eta_R| \le |e(t)| \in L^\infty(\R) $. Applying again the Lebesgue dominated convergence theorem we get 
 \begin{align*}
  \int_0^t \int_{\R} P_{N_4} (P_{\le N_1}  e_R \, P_{N_2} u_2 \partial_x P_{N_3}  u_3) P_{N_4} u_4
  &= \int_0^t \int_{\R}  P_{\le N_1}  e_R \, P_{N_2} u_2 \partial_xP_{N_3}  u_3 P_{N_4}^2  u_4\\
 \tendsto{R\to +\infty}  &\int_0^t \int_{\R}  P_{\le N_1}  e \, P_{N_2} u_2 \partial_xP_{N_3}  u_3 P_{N_4}^2  u_4\\
 &\qquad \qquad=I_t^3 \; ,
 \end{align*}
by  using that, for any fixed $ j \in \N$,  $P_{2j} u_i \in L^\infty_{Tx} \cap L^2_{Tx} $. 
This proves the desired result since
$$
\|e_R\|_{L^\infty_{tx}} + \|\partial_t e_R\|_{L^\infty_{tx}} \le \|e\|_{L^\infty_{tx}} + \|\partial_t e\|_{L^\infty_{tx}} \; , 
\forall R\ge 1 . 
$$
Now we extend the functions $e,u_2,u_3,u_4 $ on the whole time axis. For $ u_2, u_3, u_4$  we use  the extension operator $ \rho_T $ defined in Lemma \ref{extension}. On the other hand for $e$ we use the extension operator $ 
\tilde{\rho}_T $ defined by $ \tilde{\rho}_T(e)(t) = \eta(t) e( \mu_T(t)) $ with $ \mu_T $ defined in \eqref{defext} and $ \eta $ defined in \eqref{defeta}. This extension operator is bounded from $ W^{1,\infty}_T L^\infty_x  $ into  $W^{1,\infty}_t L^\infty_x  $ with a bound that does not depend on $ T>0 $. 
 To lighten the notations, we keep the notation $ u_i $ for   $ \rho_T(u_i ) $ and $ e$ for $  \tilde{\rho}_T(e)$.  
Fixing $ t\in ]0,T[ $ and setting $ R=N_2^\frac{3}{4} N_3 N_4^\frac{3}{4}  $, we then split $I_t$ as
\begin{align}
I_t(e, u_2,u_3,u_4) &= I_\infty(e,1_{t,R}^{high}u_2,1_t \, u_3, 1_t \, u_4) +
 I_\infty(e,1_{t,R}^{low}u_2, 1_{t,R}^{high} u_3, 1_t \, u_4) \nonumber \\
 & \quad +  I_\infty(e,1_{t,R}^{low}u_2,1_{t,R}^{low} u_3, 1_{t,R}^{high} u_4)+ I_\infty(e,1_{t,R}^{low}u_2,1_{t,R}^{low} u_3, 1_{t,R}^{low}  u_4)
 \nonumber \\
& :=I_t^{high,1} +I_t^{high,2} + I_t^{high,3}+ I_t^{low}, \label{decIt}
\end{align}
where $I_\infty(e,u_2,u_3,u_4) = \int_{\R^2}P_{N_4} (P_{N_1} e \, P_{N_2} u_2 \partial_x P_{N_3} u_3) P_{N_4} u_4$. The contribution of $I_t^{high,1}$ is estimated thanks to Lemma \ref{ihigh-lem} and H\"older and Bernstein  inequalities by
\begin{align}
I_t^{high,1} &\lesssim N_3 \|1_{t,R}^{high}\|_{L^1} \|e\|_{L^\infty_{tx}} \|P_{N_2} u_2\|_{L^\infty_{t}L^4_x} 
\|P_{N_3} u_3\|_{L^\infty_t L^2_x}\|P_{N_4} u_4\|_{L^\infty_{t}L^4_x} \nonumber \\
&\lesssim T^{1/4} (N_2^\frac{3}{4} N_3 N_4^\frac{3}{4} )^{-\frac{3}{4}}N_3 (N_2 N_4)^\frac{1}{4}  \|e\|_{L^\infty_{tx}}
\prod_{i=2}^4  
  \|P_{\sim N_i} u_i\|_{L^\infty_t L^2_x}  \nonumber \\
&\lesssim T^{1/4} (N_2\vee N_3)^{-\frac{1}{16}}  \|e\|_{L^\infty_{tx}} 
\prod_{i=2}^4 \|P_{N_i} u_i\|_{L^\infty_t L^2_x}  \label{estIthigh}
\end{align}
where we used that the frequency projectors ensure that $ N_2\vee N_4 \sim N_2\vee N_3 $. The contribution of 
$ I_t^{high,2} $ and $  I_t^{high,3} $ can be estimated in exactly the same way, using that$  \|1_{t,R}^{low} \|_{L^\infty_t} \lesssim 1 $ thanks to \eqref{low}.
To evaluate the contribution $ I_t^{low} $ we use the following decomposition :
\begin{align}
I_\infty( e,1_{t,R}^{low} &\,   u_2, u_3,  u_4)  = I_\infty( e,Q_{{ \gtrsim } N_2 N_3 N_4}( 1_{t,R}^{low} u_2),   1_{t,R}^{low}u_3,  1_{t,R}^{low}u_4)\nonumber \\
& +I_\infty( e,Q_{{ \ll }N_2 N_3 N_4} ( 1_{t,R}^{low}u_2), Q_{{ \gtrsim }N_2 N_3 N_4}  ( 1_{t,R}^{low}u_3),  1_{t,R}^{low} u_4 )\nonumber \\
& + I_\infty( e,Q_{{ \ll }N_2 N_3 N_4}  ( 1_{t,R}^{low}u_2), Q_{{ \ll }N_2 N_3 N_4} ( 1_{t,R}^{low} u_3), Q_{\gtrsim N_2 N_3 N_4}( 1_{t,R}^{low}u_4))
\nonumber \\
& + I_\infty( e,Q_{{ \ll }N_2 N_3 N_4}  ( 1_{t,R}^{low}u_2), Q_{{ \ll }N_2 N_3 N_4} ( 1_{t,R}^{low} u_3), Q_{{ \ll }N_2 N_3 N_4}  ( 1_{t,R}^{low} u_4))\nonumber\\
&=I_t^{2,low}+I_t^{3,low}+I_t^{4,low}+I_t^{1,low}
\label{AA}\; ,
\end{align}
To evaluate the  contribution $ I_t^{1,low} $ we notice that since $ N_1^3 \ll N_1 N_2 N_3 $, Lemma \ref{resolem}  ensures that 
$$
I_t^{1,low}= I_\infty( R_{\sim N_2 N_3 N_4} e,Q_{{ \ll }N_2 N_3 N_4}  ( 1_{t,R}^{low}u_2), Q_{{ \ll }N_2 N_3 N_4} ( 1_{t,R}^{low} u_3), Q_{{ \ll }N_2 N_3 N_4}  ( 1_{t,R}^{low} u_4))
$$ 
where $ R_K $ is the projection on the time Fourier variable (see \eqref{project}). Therefore, 
by  Bernstein inequality and  Lemma \ref{continuiteQ} we get
\begin{align}
| I_t^{1,low}| &\lesssim  T (N_2 N_3 N_4)^{-1} \| e_t\|_{L^\infty_{tx}} \|P_{N_2}  u_2 \|_{L^\infty_{tx}}  
N_3  \|P_{N_3} u_3  \|_{L^\infty_t L^2_x}  \|P_{N_4}  u_4 \|_{L^\infty_t L^2_x}  \nonumber \\
 &\lesssim T (N_2\vee N_3)^{-\frac{1}{2}} \|e_t\|_{L^\infty_{tx}} \|P_{N_2} u_2\|_{L^\infty_t L^2_x} \prod_{i=3}^4
  \|P_{N_i} u_i\|_{L^\infty_t L^2_x} 
\end{align}
Now, to evaluate the other contributions in \eqref{AA} we have to separate  different cases. For the future use of Lemma \ref{ilow-lem}, it is worth noticing that since $ N_2,N_4 \gg 1  $,  $R=N_2^\frac{3}{4} N_3 N_4^\frac{3}{4}  \ll N_2 N_3 N_4 $. \\
{\it Case 1 : $N_4\sim N_3\gtrsim N_2$}.  
Then  $ I_t^{2,low}$  can be easily estimated thanks to  
 Lemma \ref{ilow-lem}  and \eqref{low} by 
\begin{align}
| I_t^{2,low}|    
\lesssim &\|e\|_{L^\infty_{tx}}  \|Q_{{ \gtrsim } N_2 N_3 N_4} P_{N_2} (1_{t,R}^{low}u_2) \|_{L^2_{tx}}
 N_3 \| 1_{t,R}^{low}P_{N_3}  u_3\|_{L^2_{tx}}\| 1_{t,R}^{low}P_{N_4} u_4 \|_{L^\infty_{tx}}  \nonumber\\
 \lesssim & T^{1/2} (N_2 N_3 N_4)^{-1} N_2 N_3 N_4^\frac{1}{2} \|e\|_{L^\infty_{tx}}\| P_{N_2} u_2 \|_{X^{-1,1}}
  \|P_{N_3} u_3\|_{L^\infty_t L^2_{x}} \|P_{N_4} u_4 \|_{L^\infty_t L^2_{x}}  \nonumber\\
 \lesssim & T^\frac{1}{2} (N_2\vee N_3)^{-1/2}  \|e\|_{L^\infty_{tx}}\|  u_2 \|_{X^{-1,1} } \prod_{i=3}^4
  \| P_{N_i} u_i\|_{L^\infty_t L^2_x} 
\end{align}
To estimate the contribution of $I_t^{3,low} $ we notice that Lemma \ref{ihigh-lem}  together with the fact that $ R\ge N_2\vee N_3  $  ensure that for any $ w \in L^\infty_t L^2_x $ 
$$
\| 1_{t,R}^{low} w \|_{L^2_{tx}} \le \| 1_{t} w \|_{L^2_{tx}}+ \| 1_{t,R}^{high} w \|_{L^2_{tx}}  
\lesssim \|w\|_{L^2_T L^2_x} + T^{1/4} (N_2\vee N_3)^{-1/4} \|w\|_{L^\infty_T L^2_x} \; .
$$
Therefore     Lemmas \ref{continuiteQ} and \ref{ilow-lem} lead to 
  \begin{align}
 |I_t^{3,low}| \lesssim &   (N_2 N_3 N_4)^{-1} N_3^2  \|e\|_{L^\infty_{tx}} \|P_{N_2} u_2 \|_{L^\infty_{tx}}
  \|P_{N_3} u_3\|_{X^{-1,1}}\|1_{t,R}^{low} P_{N_4} u_4 \|_{L^2_{tx}}  \nonumber\\
 \lesssim & N_2^{-1/2} \|e\|_{L^\infty_{tx}}\|P_{N_2} u_2 \|_{L^\infty_t L^2_x}
 \Bigl( \|P_{N_3} u_3\|_{X^{-1,1}}\|P_{N_4} u_4\|_{L^2_T L^2_x}  \nonumber \\
  & \quad + T^{1/4} (N_2\vee N_3)^{-1/4}     
 \|P_{N_3} u_3\|_{X^{-1,1}}  \|P_{N_4} u_4\|_{L^\infty_T L^2_x}\Bigr)
\end{align}
and  $ I_t^{4,low}$ can be estimated in exactly the same way by exchanging the role of $ u_3 $ and $ u_4$ to get 
\begin{eqnarray}
| I_t^{4,low}| &\lesssim & N_2^{-1/2} \| e \|_{L^\infty_{tx}} \|P_{N_2} u_2 \|_{L^\infty_t L^2_x}
\Bigl( \|P_{N_3} u_4\|_{X^{-1,1}}\|P_{N_4} u_3\|_{L^2_T L^2_x}  \nonumber \\
  && \quad + T^{1/4} (N_2\vee N_3)^{-1/4}     
 \|P_{N_3} u_4\|_{X^{-1,1}}  \|P_{N_4} u_3\|_{L^\infty_T L^2_x}\Bigr) \label{AA5}
\end{eqnarray}
Gathering \eqref{decIt}-\eqref{AA5}, we obtain \eqref{f1} whenever $N_4\sim N_3\gtrsim N_2$.\\
{\it Case 2 : $N_2\sim N_3\gtrsim N_4$}. Then we get exactly the same type of estimates just by exchanging the role of $ u_2 $ and $ u_4 $ with respect to the preceding case. \\
{\it Case 3: $N_2\sim N_4\gtrsim N_3$}. This case can be treated as the first ones and is even simplest since the derivative falls on the smallest frequency. We thus omit the details.
\end{proof}
\subsection{A priori estimates in $ H^s(\R) $}
For an initial data in $ H^s(\R) $, with $ s>1/2$, we will construct a solution to \eqref{KdV2} in $Y^s_T $ whereas the estimate of difference of two solutions emanating from initial  data belonging to $ H^s(\R) $ will take place in $ Y^{s-1}_T $.
\begin{lemma} \label{estYs}
Let $ s>1/2 $, $0<T<1$  and $ u \in L^\infty_T H^s\cap  L^2_{[b]}(]0,T[;H^{s+1})$ be a solution to \eqref{KdV2}. Then $u\in 
Y^{s}_T $ and the following inequality holds
\begin{equation}\label{esta1}
\|u\|_{Y^{s}_T} \lesssim C \Bigl( \|u\|_{L^2_{[b]}(]0,T[;H^{s+1})}+(1+\|u\|_{L^\infty_T H^{\frac{1}{2}+}}) \, \| u \|_{L^\infty_T H^s}\Bigr)
\; .
\end{equation}
Moreover, for any couple $(u, {v}) \in   L^\infty_T H^s $ of solutions to \eqref{KdV2} associated with a couple of initial data 
 $ (u_0,v_0)\in (H^s(\R))^2 $, it holds
\begin{equation}\label{estdiffXregular}
\|u-v\|_{Y^{s-1}_T}  \lesssim 
C \Bigl(   \|u-v\|_{L^2_{[b]}(]0,T[;H^{s})}+ (1+ \|u+v\|_{L^\infty_T H^{s}})  \|u-v\|_{L^\infty_T H^{s-1}}\Bigr) \; ,
\end{equation}
where
$$
C=C \Bigl( s, \|b\|_{L^\infty_T C_*^{((s+1)\vee 2)+}}, \|c\|_{L^\infty_T C_*^{s+}}, 
\|d\|_{L^\infty_T C_*^{s+}}, \|e\|_{L^\infty_T C_*^{(s\vee 1 )+}}, \|f\|_{L^\infty_T C_*^{s+}}\Bigr)\; .
$$
\end{lemma}
\begin{proof}
According to the extension Lemma \ref{extension} it suffices to establish estimates on the Bourgain's norms of $ u $ and $u-v$.
 Standard linear estimates in Bourgain's spaces lead to
  \begin{align*}
  \|u\|_{X^{s-1,1}_T} & \lesssim   \|u_0\|_{H^{s-1}}+\|1_T \, (\partial_t-\partial_x^3) u  \|_{X^{s-1,0}} \\
  & \lesssim   \|u_0\|_{H^{s-1}}+ \|b u_x \|_{L^2_T H^s}+  \|b_x u_x  \|_{L^2_T H^{s-1}}+\|cu  \|_{L^2_T H^s} \\
 & \hspace*{0.5cm} + \| (-c_x+d)u\|_{L^2_T H^{s-1}} + \frac{1}{2}\|e \, u^2 \|_{L^2_T H^s} +\| (-e_x/2+f) u^2 \|_{L^2_T H^{s-1}} \: .
 \end{align*}
 According to Lemma \ref{product}, using that $s>1/2$,  it holds 
 $$
 \|b_x u_x \|_{L^2_T H^{s-1}} \lesssim \|b_x\|_{L^\infty_T C^{|s-1|+}_*} \|u_x\|_{L^\infty_T H^{s-1}}
 $$
 $$
 \|cu  \|_{L^2_T H^s}+\|c_x u  \|_{L^2_T H^{s-1}}\lesssim \|c\|_{L^\infty_T C^{s+}_*} \|u\|_{L^\infty_T H^{s}}
 $$
 $$
 \|e u^2  \|_{L^2_T H^{s}}+\|e_x u^2  \|_{L^2_T H^{s-1}}\lesssim \|e\|_{L^\infty_T C^{s+}_*} \|u\|_{L^\infty_T H^{\frac{1}{2}+}}
 \|u\|_{L^\infty_T H^{s}}
 $$
 $$
 \|d u \|_{L^2_T H^{s-1}}\lesssim \|d\|_{L^\infty_T C^{|s-1|+}_*} \|u\|_{L^\infty_T H^{s}} \text{ and } \|f u^2\|_{L^2_T H^{s-1}}\lesssim \|f\|_{L^\infty_T C^{|s-1|+}_*} \|u\|_{L^\infty_T H^{s}}^2
 $$
 Therefore, we get 
 $$
 \|u\|_{X^{s-1,1}_T} \lesssim  { \|u\|_{L^\infty_T H^{s-1}}}+ C_1  (1+\|u\|_{L^\infty_T H^{\frac{1}{2}+}}) \| u \|_{L^\infty_T H^s} + \|b u_x \|_{L^2_T H^s} \; ,
 $$
 where $ C_1=C_1(\|b_x\|_{L^\infty_T C_*^{|s-1|+}}, \|c\|_{L^\infty_T C_*^{s+}}, \|d\|_{L^\infty_T C_*^{|s-1|+}},\|e\|_{L^\infty_T C_*^{s+}}, \|f\|_{L^\infty_T C_*^{|s-1|+}})$.\\
 Now, noticing that Lemma \ref{product} also leads for $ s>1/2$  to 
 $$
 \|b_x w_x \|_{L^2_T H^{s-2}} \lesssim \|b_x\|_{L^\infty_T C^{|s-2|+}_*} \|w_x\|_{L^\infty_T H^{s-2}}
 $$
 $$
 \|cw  \|_{L^2_T H^{s-1}}+\|c_x w  \|_{L^2_T H^{s-2}}\lesssim \|c\|_{L^\infty_T C^{s+}_*} \|w\|_{L^\infty_T H^{s-1}}
 $$
 $$
 \|e u w  \|_{L^2_T H^{s-1}}+\|e_x u w   \|_{L^2_T H^{s-2}}\lesssim \|e\|_{L^\infty_T C^{( s\vee (2-s))+}_*} \|u\|_{L^\infty_T H^{s}}
 \|w\|_{L^\infty_T H^{s-1}}
 $$
 $$
 \|d w \|_{L^2_T H^{s-2}}\lesssim \|d\|_{L^\infty_T C^{s+}_*} \|w\|_{L^\infty_T H^{s-1}} \text{ and } \|f u^2\|_{L^2_T H^{s-2}}\lesssim \|f\|_{L^\infty_T C^{s+}_*} \|w\|_{L^\infty_T H^{s-1}}^2\; ,
 $$
we also get 
\begin{align*}
  \|u-v\|_{X^{s-1,1}_T}&
   \lesssim  \|u_0-v_0\|_{H^{s-1}}+C_2 \; (1+\|u+v\|_{L^\infty_T H^{s}}) \| u-v \|_{L^\infty_T H^{s-1}}\\ 
  & \hspace*{1cm}+ \|b\,  \partial_x (u-v) \|_{L^2_T H^{s-1}} 
\end{align*}
 with   $ C_2=C_2(  \|b_x\|_{L^\infty_T C_*^{|s-2|+}}, \|c\|_{L^\infty_T C_*^{s+}}, \|d\|_{L^\infty_T C_*^{s+}}
 +\|e\|_{L^\infty_T C_*^{( s\vee (2-s))+}}, \|f\|_{L^\infty_T C_*^{s+}})$.\\
  It just remains to get an estimate on $\| \partial_x(b v_x)\|_{L^2_T H^{\theta-1}}$ for $ b \in L^\infty_T C_\ast^{(s \vee (3-s))+}$ and $ v\in L^\infty_T H^\theta $ with $ \theta>-1/2 $.  
By using a non homogeneous dyadic decomposition it holds 
$$  \|\partial_x (b v_x)\|_{L^2_T H^{\theta-1}}^2 \sim 
\| \partial_x P_{\lesssim 1}1(b v_x) \|_{L^2_T L^2_x}^2+ \sum_{N\gg 1  } N^{2\theta} \| P_N(b v_x) \|_{L^2_T L^2_x}^2 \; . $$
The first term of the above right-hand side is easily estimated as above by :
$$
\| \partial_x P_{\lesssim 1} (b v_x) \|_{L^2_T L^2_x}\lesssim \|b v_x\|_{H^{-2}} \lesssim  \|b\|_{L^\infty_T C^{\frac{3}{2}+}_*} \|
v\|_{L^\infty_T H^{-\frac{1}{2}+}}
$$
Now, for $ N\gg 1  $ we rewrite $ P_N(b u_x)$ as 
\begin{align*}
P_N(b u_x)& = P_N(P_{\gtrsim N} b \, u_x)  +P_{\ll N} b P_N u_x +[P_N,P_{\ll N}b] u_x \\
& = A_N+B_N +C_N .
\end{align*}
We have 
\begin{align} 
\sum_{N\gg 1}N^{2\theta} \| A_N \|_{L^2_{Tx}}^2 & \lesssim \sum_{N\gg 1}
 N^{2\theta} \| P_{\gtrsim N} b 
P_{\ll N} u_x \|_{L^2_{Tx}}^2+ \sum_{N\gg  1} N^{2\theta}  \sum_{N_1\gtrsim N}  \| P_{N_1} b 
P_{\sim N_1} u_x \|_{L^2_{Tx}}^2\nonumber \\
&  \lesssim \sum_{N\gg 1}
 N^{2\theta} \| P_{\gtrsim N} b \|_{L^\infty_{Tx}} N^{(2-2\theta)\vee 0}
\|P_{\ll N} u_x \|_{L^2_T H^{\theta-1}}^2\nonumber \\ 
& \quad + \sum_{N\gg 1}N^{2\theta}  \sum_{N_1\gtrsim N} \| P_{N_1} b  \|_{L^\infty_{Tx}}N_1^{2-2\theta}
\|P_{\sim N_1} u_x \|_{L^2_T H^{\theta-1}}^2\nonumber \\
 & \lesssim \Bigl( \|b\|_{L^\infty_T C_\ast^{(1 \vee \theta)+}}^2 + \|b\|_{L^\infty_T C_\ast^{(1 \vee (1- \theta))+}}^2
 \Bigr) \|u\|_{L^\infty_T H^{\theta}}^2 
  \; .
 \end{align}
 To bound the contribution of $ B_N $ we observe that 
 \begin{align} 
\sum_{N\gg  1} N^{2\theta} \| B_N \|_{L^2_T L^2_x}^2 & \le \sum_{N\gg  1} N^{2\theta}
\Bigl( \| b \partial_x u_N\|_{L^2_T L^2_x}+  \| P_{\gtrsim N} b \partial_x u_N\|_{L^2_T L^2_x}\Bigr)^2\nonumber \\
& \lesssim \sum_{N\gg  1} N^{2\theta}\int_{0}^T \int_{\R} 
 b^2 (\partial_x u_N)^2 +\sum_{N\gg  1} N^{2} \|P_{\gtrsim N} b \|_{L^\infty_{Tx}}^2 \|u_N \|_{L^2_{T}H^\theta}^2
 \nonumber \\
 &  \le \|u \|_{(L^2_T H^{\theta+1})_b}^2\;  + \|b_x\|_{L^\infty_{Tx}}^2 \|u\|_{L^\infty_T H^\theta}^2 \; .
 \end{align}
 
Finally to bound the contribution of $ C_N $  we use \eqref{commu} of Lemma \ref{commutator} to get
 \begin{equation}
 \sum_{N\gg 1}N^{2\theta} \|C_N \|_{L^2_T L^2_x}^2  \lesssim \|b_x\|_{L^\infty_{Tx}}^2 
 \sum_{N\gg 1 } N^{2\theta} \|\tilde{P}_N u \|_{L^2_T L^2_x}^2 \lesssim  \|b_x\|_{L^\infty_{Tx}}^2  \|u\|_{L^\infty_T H^\theta}^2 \quad .
   \end{equation}
   Gathering the above estimates we observe that it is enough to have  $ b\in L^\infty_T C_*^{(3-s)+} $ for $ 1/2<s<3/2 $ 
    and   $ b\in L^\infty_T C_*^{s+} $ for $ s\ge 3/2 $. 
   and completes the proof of the lemma.
\end{proof}

\begin{proposition}\label{prou}
Let $ 0<T<2$ and $ u\in Y^s_T$ with $ s>1/2 $ be a solution to \eqref{KdV2}  {associated with an initial datum 
 $ u_0\in H^s(\R) $}. Then it holds
\begin{equation}\label{estHsregular}
\|u\|_{L^\infty_TH^s}^2 + \|u\|_{L^2_{[b]}(]0,T[;H^{s+1})}^2 \le  \|u_0\|_{H^s}^2 + C \; T^\frac{1}{16} (1+\|u\|_{Y^{\frac{1}{2}+}_T}) \|u\|_{Y^s_T}^2 \;.
\end{equation}
where 
\begin{equation}\label{const}
C=C \small{\Bigl( s, \|b\|_{L^\infty_T C_*^{((s+1)\vee 2)+}}, \|c\|_{L^\infty_T C_*^{(s\vee 1)+}}, 
\|d\|_{L^\infty_T C_*^{s+}}, \|e\|_{L^\infty_T C_*^{s+\frac{1}{2}+}}, \|f\|_{L^\infty_T C_*^{s+}}, \|e_t\|_{L^\infty_{Tx}}\Bigr)}
\end{equation}
\end{proposition}
\begin{proof}
 We apply  the operator $ P_N $ with $ N\in 2^{\N}  $ dyadic  to equation \eqref{KdV2}.  On account of Remark \ref{rem1}, it is clear that
 $ P_N u \in C([0,T];H^\infty) $ with $\partial_t u_N \in L^\infty(0,T;H^\infty)$. Therefore , taking the $L^2_x $-scalar product of the resulting equation with $ P_N u $, multiplying by
  $ \langle N \rangle^{2s} $ and integrating
 on $ ]0,t[ $ with $ 0<t<T $  we obtain
\begin{align}
 \langle N \rangle^{2s}  \|P_N u(t)\|_{L^2}^2 & =  \langle N \rangle^{2s} \|P_N u_0\|_{L^2}^2 +
   \langle N \rangle^{2s} \int_0^t \int_{\R}  P_N \Bigl(-b_x u_x -c u_x-du+f u^2\Bigr) P_N u  
 \nonumber \\
 & 
 +   \langle N \rangle^{2s}\Bigl(  \int_0^t \int_{\R} P_N(e\, u u_x) P_N u +
   \int_0^t \int_{\R}  \partial_x P_N (b u_x) P_N u \; \Bigr) \label{estenergy} \; .
 \end{align}
 
Now we are going to estimate successively all the terms of  the right-hand of \eqref{estenergy}. Note that, even if $ s>1/2$,   we will  give estimates of the linear terms (in $u$)  valid for $ s>-1/2 $ that will be directly usable in Proposition \ref{prodif} when estimating the difference of two solutions  in $ H^{s-1}(\R)$. \\
$ \bullet $ {\it Contribution of $ P_N(du) $}.\\
Making use of Sobolev inequalities, this contribution is easily estimated  by :
\begin{align}
\cro{N}^{2s} \Bigl| \int_{]0,t[\times \R}  P_N (du)    P_N u \Bigr| 
&  \lesssim \cro{N}^{2s}   \|P_N (d u ) \|_{L^2_T L^2_x}    \|P_N u \|_{L^2_T L^2_x}  \nonumber\\
& \lesssim  T \delta_N \|d u\|_{L^\infty_T H^s}\|u\|_{L^\infty_T H^s}\nonumber\\
& \lesssim T \delta_N \, \|d\|_{L^\infty_T C_*^{|s|+}}  \|u\|_{L^\infty_T H^s}^2
\end{align}
with $ \|(\delta_{2^j})_{j\ge 0}\|_{l^1} \le 1 $. 
 In the sequel, we denote by $ (\delta_q)_{q\ge 1} $ any sequence of real numbers such that $  \|(\delta_{2^j})_{j\ge 0}\|_{l^1} \le 1 $.
\vspace*{2mm}  \\
$ \bullet $ {\it Contribution of $ P_N(f u^2) $}.\\
This term is only estimated for $ s>1/2 $. Proceding exactly as above we get
\begin{align}
\cro{N}^{2s} \Bigl| \int_{]0,t[\times \R}  P_N (fu^2)\Bigr)    P_N u \Bigr| 
& \lesssim  T \delta_N \|f u^2\|_{L^\infty_T H^s}\|u\|_{L^\infty_T H^s}\nonumber\\
& \lesssim T \delta_N \, \|f\|_{L^\infty_T C_*^{|s|+}} (1+\|u\|_{L^\infty_{Tx}} )\|u\|_{L^\infty_T H^s}^2\; .
\end{align}
$ \bullet $ {\it Contribution of $ P_N(( b_x+c) u_x) $}.\\
For $ 1\le N\lesssim  1 $, \eqref{estCs} leads to 
\begin{align*}
  \langle N\rangle^{2s} \Bigl| \int_0^t\int_{\R}  \Bigl(P_N((b_x+c) u_x) \Big) P_N u \Bigr|
  & \lesssim \int_0^t  \|(b_x+c) u_x\|_{H^{s-1}} \|u\|_{H^s} \\
  &\lesssim  \Bigl(\|b_x\|_{L^\infty_T C_*^{|s-1|+}} +\|c\|_{L^\infty_T C_*^{|s-1|+}} \Bigr) \|u\|_{H^s}^2 \; .
  \end{align*}
For $ N\gg 1 $,  We first notice that 
 \begin{align}
N^{2s} \Bigl| \int_{]0,t[\times \R} P_N \Bigl( P_{\gtrsim N} (b_x+c)u_x\Bigr)    P_N u \Bigr| & \nonumber\\
\lesssim  N^{2s}\Bigl| \int_{]0,t[\times \R}P_N \Bigl( P_{\gtrsim N} &(b_x+c)P_{\ll N} u_x\Bigr)    P_N u \Bigr| \nonumber \\
 +  N^{2s} \sum_{N_1\gtrsim N} & \Bigl| \int_{]0,t[\times \R} P_N \Bigl( P_{N_1} (b_x+c)P_{\sim N_1} u_x\Bigr)    P_N u \Bigr|
\nonumber \\
\lesssim  \int_0^t N^{s}\| P_{\gtrsim N} (b_x+c) \|_{L^\infty_x} &  N^{(1-s)\vee 0}\|P_{\ll N} u_x\|_{H^{s-1}} \|u\|_{H^s} \nonumber \\
 +   \int_0^t  N^{s}  \sum_{N_1\gtrsim N}  \| P_{N_1}  & (b_x+c)\|_{L^\infty_x}  N_1^{1-s} \|P_{\sim N_1} u_x\|_{H^{s-1}} \|u\|_{H^s}
\nonumber \\
 \lesssim T \delta_N \,  (\|b_{x}\|_{L^\infty_{T}C^{(1\vee s\vee1-s)+}_\ast}  & + \|c\|_{L^\infty_{T}C^{(1\vee s\vee1-s)+}_\ast}) \|u\|^2_{L^\infty_T H^s}
\end{align}
Then  we use the commutator estimate \eqref{commu} and integration by parts to get 
 \begin{align}
N^{2s} \Bigl| \int_{]0,t[\times \R} P_N \Bigl( P_{\ll N} (b_x+c)u_x\Bigr)&     P_N u \Bigr| =  
N^{2s}  \Bigl| \int_{]0,t[\times \R} P_{\ll N}(b_{xx}+c_x) (P_N u)^2 \Bigr| \nonumber \\
& +  N^{2s}  \Bigl| \int_{]0,t[\times \R} [P_N,P_{\ll N}(b_x+c)] u_x   P_N u \Bigr|\nonumber \\
& \lesssim N^{2s} \| b_{xx}+c_x\|_{L^\infty_{Tx}} \|\tilde{P}_N u \|_{L^2_T L^2_x}^2
 \nonumber \\
& \lesssim T \delta_N \,  (\|b_{xx}\|_{L^\infty_{Tx}}+ \|c_x\|_{L^\infty_{Tx}}) \|u\|^2_{L^\infty_T H^s}
\end{align}
with $ \|(\delta_{2^j})_{j\ge 0}\|_{l^1} \le 1 $.\vspace*{2mm}  \\
%%%%%%%%%%%%%%%%%%%%%%%%%%%%%
\noindent
$ \bullet $ {\it Contribution of $ P_N(e u u_x) $}.\\
This term is only estimated for $ s>1/2 $. For $ 1\le N\lesssim 1 $, we write $ e\,   \partial_x(u^2) = \frac{1}{2} \partial_x (eu^2)-\frac{1}{2} e_x u^2 $ to get 
\begin{align}
\cro{N}^{2s} \Bigl| \int_{]0,t[\times \R} P_N \Bigl( e   \partial_x(u^2) \Bigr) P_N u \Bigr| &
\lesssim  T \|u\|_{L^\infty_T L^2_x} (\| e_x u^2 \|_{L^\infty_T L^2_T} +\|e u^2 \|_{L^\infty_T L^2_x} ) \nonumber \\
& \lesssim \,  T \| e\|_{L^\infty_T W^{1,\infty}} \, \|u\|_{L^\infty_T H^{\frac{1}{4}}}^2 \|u\|_{L^\infty_T L^2_x}
\end{align}
It thus remains to consider $ N\gg 1$.
We first separate   two contributions. \\
{\bf 1.} The contribution of $ P_N(P_{\gtrsim N} e \, u u_x) $.  This contribution is easily estimated by 
\begin{align}
N^{2s} \Bigl| \int_{]0,t[\times \R}  P_N & \Bigl( P_{\gtrsim N} e\,   \partial_x(u^2) \Bigr) P_N u \Bigr| \nonumber\\
& =  N^{2s} \sum_{N_1\ll N}  \Bigl| \int_{]0,t[\times \R} P_N \Bigl( P_{\sim N} e \, P_{N_1}\partial_x(u^2) \Bigr) P_N u \Bigr| \nonumber \\
& +  N^{2s} \sum_{N_1\gtrsim N}  \Bigl| \int_{]0,t[\times \R} P_N \Bigl( P_{\sim N_1} e \, P_{N_1}\partial_x(u^2) \Bigr) P_N u \Bigr|\nonumber \\
& \lesssim N^{2s}\int_{0}^t  \| P_{\sim N}e \|_{L^\infty_x} \|P_N u\|_{L^2_x} \sum_{N_1\ll N} N_1^{1/2} \| D_x^{1/2}(u^2) \|_{L^2_x} \nonumber \\
&+ N^{2s}\int_{0}^t  \|P_N u\|_{L^2_x}\sum_{N_1\gtrsim N}  \| P_{\sim N_1} e \|_{L^\infty_x}  N_1^{1/2} \| D_x^{1/2}(u^2) \|_{L^2_x} \nonumber \\
& \lesssim \delta_N \,  T \| e\|_{L^\infty_T C_\ast^{s+1/2}} \, \|u\|_{L^\infty_T H^{\frac{1}{2}+}}^2 \|u\|_{L^\infty_T H^s}
\end{align}
with $ \|(\delta_{2^j})_{j\ge 0}\|_{l^1} \le 1 $.

{\bf 2.} The contribution of $ P_N(P_{\ll N}  e\,  u u_x) $.
We rewrite this term as 
\begin{align}
 P_N(P_{\ll N} e\,  u u_x)   = 
 & P_N\Bigl( P_{\ll N} e  \, P_{\lesssim 1} u \,  \tilde{P}_N (u_x)\Bigr) \nonumber \\
 &+\sum_{1\ll N_2\ll N} P_N\Bigl( P_{\ll N}e \,  u_{N_2} \tilde{P}_N u_x \Bigr) \nonumber \\
  &+P_N\Bigl( P_{\ll N}e \,  \tilde{P}_N u \,  P_{\lesssim 1}u_x\Bigr) \nonumber \\
 & + \sum_{1\ll N_3\lesssim N_1\ll N } P_N\Bigl( e_{N_1}  \tilde{P}_N u \,  \partial_x u_{N_3} \Bigr) \nonumber \\
  & + \sum_{  1\ll N_3\lesssim N_2} P_N\Bigl( P_{\ll N_3\wedge N} e \, u_{N_2} \partial_x u_{N_3} \Bigr)
  \nonumber \\
  =& A+B+C+D+E. 
\end{align}
First, the contribution of $C $ is easily estimated by 
\begin{align}
   N^{2s} \Bigl| \int_{]0,t[\times \R} C P_N u \Bigr| 
& \lesssim \int_0^t \|u_N\|_{H^s} \|u_{\sim N} \|_{H^s} \| e \|_{L^\infty_{x}} \|u\|_{L^\infty_T L^2_x} \nonumber \\
& \lesssim \delta_N \, T \| e\|_{L^\infty_{Tx}} \, \|u\|_{L^\infty_T L^2_x} \|u\|_{L^\infty_T H^s}^2 \label{td1}
\end{align}
with $ \|(\delta_{2^j})_{j\ge 0}\|_{l^1} \le 1 $.\\
The contribution of $ D $ is  estimated in the following way :
\begin{align}
   N^{2s} \Bigl| \int_{]0,t[\times \R} & D  P_N u \Bigr| =  
N^{2s} \Bigl|\sum_{1\ll N_3\lesssim N_1\ll N}     \int_{]0,t[\times \R}  P_N\Bigl( e_{N_1}  u_{\sim N} \partial_x u_{N_3} \Bigr)  P_N u \Bigr| \nonumber \\
& \lesssim \int_0^t \|u_N\|_{H^s} \|u_{\sim N} \|_{H^s} \sum_{1\le  N_1\ll N} \|e_{N_1} \|_{L^\infty_{x}} \sum_{N_3\lesssim N_1} N_3  N_3^{0-} \|u_{N_3} \|_{H^{\frac{1}{2}+}} \nonumber \\
& \lesssim \|u_N\|_{L^2_T H^s}^2 \|e \|_{L^\infty_T C_\ast^{1}}  \|u\|_{L^\infty_T H^{\frac{1}{2}+}} \nonumber \\
& \lesssim \delta_N \, T^{1/2} \| e\|_{L^\infty_T C_\ast^{1}} \, \|u\|_{L^\infty_T H^{\frac{1}{2}+}}\|u\|_{L^\infty_T H^s}^2 
\label{d1}
\end{align}
with $ \|(\delta_{2^j})_{j\ge 0}\|_{l^1} \le 1 $.

To bound the contribution of $ A $   we use the  commutator estimate \eqref{commu}  and integration  by parts to get   \begin{align}
 N^{2s} \Bigl| \int_{]0,t[\times \R} A  P_N u \Bigr|&  \lesssim  
N^{2s} \Bigl| \int_{]0,t[\times \R}  \partial_x  (P_{\ll N}  e  P_{\lesssim 1} u)  (P_N u)^2 \Bigr| 
\nonumber \\
&+ N^{2s}  \sum_{N_1\ll N, N_2\lesssim 1} \Bigl| \int_{]0,t[\times \R} [P_N, P_{\ll N}  e  P_{\lesssim 1} u] \tilde{P}_N u_x P_N u \Bigr| \nonumber \\
&\lesssim T N^{2s} \|\partial_x(P_{\ll N}  e  P_{\lesssim 1} u)\|_{L^\infty_{Tx}} 
\|\tilde{P}_N u \|_{L^\infty_T L^2_x}^2 \nonumber \\
& \lesssim \delta_N \, T  \| e\|_{L^\infty_T C_\ast^{1+}} \, \|u\|_{L^\infty_T L^2_x } \|u\|_{L^\infty_T H^s}^2 \label{cco}
\end{align}
with $ \|(\delta_{2^j})_{j\ge 0}\|_{l^1} \le 1 $.

To bound  the contribution of $ E $ , we notice that  the integral is of the form \eqref{II} so that we can use Lemma \ref{lemtriest}.
We separate the contribution $ E_1 $ of the sum over  $ N_2\sim N_3\gtrsim N $
 and the contribution $ E_2$ of the sum over $N_2\sim N \gg N_3 $. 
 For the first contribution, Lemma \ref{lemtriest} leads to 
\begin{align}
   N^{2s} & \Bigl| \int_{]0,t[\times \R} E_1  P_N u \Bigr|  \nonumber \\
  & \lesssim \sum_{N_2\gtrsim N } 
     (\|e\|_{L^\infty_{Tx}}+\|e_t \|_{L^\infty_{Tx}})\Bigl[  \|P_{N} u\|_{L^\infty_T L^2_x} 
 \|P_{\sim N_2} u \|_{L^2_T H^s}   \|P_{\sim N_2} u \|_{X^{s-1,1}_T}\nonumber \\
    &\hspace{0.5cm} + T^\frac{1}{16} N_2^{-\frac{1}{4}}   \Bigl( \|P_{N}u \|_{X^{-1,1}_T}+ \|P_{N} u \|_{L^\infty_T L^2_x} \Bigr) 
   \| P_{\sim N_2}u\|_{L^\infty_T H^s}^2 \nonumber\\
    &\hspace{0.5cm} + T^\frac{1}{16} N_2^{-\frac{1}{4}}   \Bigl( \|P_{\sim N_2}u \|_{X^{s-1,1}_T}+ \|P_{\sim N_2} u \|_{L^\infty_T H^s} \Bigr) 
   \| P_{\sim N_2}u\|_{L^\infty_T H^s}  \|P_{N} u \|_{L^\infty_T L^2_x} \Bigr] \nonumber\\
    & \lesssim  N^{-(0+)} \;  (\|e\|_{L^\infty_{Tx}}+\|e_t \|_{L^\infty_{Tx}})\Bigl(  \|u\|_{L^\infty_T  H^{0+}} 
    \|u\|_{X^{s-1,1}_T} \|u\|_{L^2_T H^s} \nonumber \\
    &\hspace{7cm}+  T^\frac{1}{16} \|u\|_{Y^0_T}  \|u\|_{Y^s_T}^2\Bigr) \nonumber \\
  & \lesssim T^\frac{1}{16}N^{-(0+)}  \;  (\|e\|_{L^\infty_{Tx}}+\|e_t \|_{L^\infty_{Tx}}) \|u\|_{Y^{0+}_T}\|u\|_{Y^{s}_T}^2 \; .
  \label{E1}
\end{align}
 In the same way 
 Lemma \ref{lemtriest} leads to 
\begin{align}
   N^{2s} \Bigl| & \int_{]0,t[\times \R}  E_2  P_N u \Bigr| \nonumber \\
  & \lesssim \sum_{  1\ll N_3\ll N } 
     (\|e\|_{L^\infty_{Tx}}+\|e_t \|_{L^\infty_{Tx}})\Bigl[  \|P_{N_3} u\|_{L^\infty_T L^2_x} 
 \|P_{\sim N} u \|_{L^2_T H^s}   \|P_{\sim N} u \|_{X^{s-1,1}_T}\nonumber \\
    & \hspace*{6cm}+ T^\frac{1}{16} N^{-\frac{1}{4}}  \|u\|_{Y^0_T} \|u\|_{Y^s_T}^2\Bigr] \nonumber\\
    & \lesssim  \delta_N \;  (\|e\|_{L^\infty_{Tx}}+\|e_t \|_{L^\infty_{Tx}}) \|u\|_{Y^{0+}_T} \Bigl(
    \|u\|_{X^{s-1,1}_T} \|u\|_{L^2_T H^s} +  T^\frac{1}{16}  \|u\|_{Y^s_T}^2 \Bigr) \nonumber \\
  & \lesssim T^\frac{1}{16}   \delta_N  \;  (\|e\|_{L^\infty_{Tx}}+\|e_t \|_{L^\infty_{Tx}}) \|u\|_{Y^{0+}_T}\|u\|_{Y^{s}_T}^2 \; ,
  \label{E2}
\end{align}
with $ \|(\delta_{2^j})_{j\ge 0}\|_{l^1} \le 1 $.

Finally  we rewrite $ B $ as 
\begin{align}
B & =\sum_{1\ll N_2\ll N } P_N\Bigl( P_{\ll N_2} e \,   u_{N_2} \tilde{P}_N (u_x)\Bigr) + \sum_{1\ll N_2\ll N }  P_N\Bigl( P_{\ll N} P_{\gtrsim N_2}e\,  u_{N_2} \tilde{P}_N (u_x)\Bigr)\nonumber \\
& = B_1+B_2 \; .
\end{align}
We notice that the integral in the contribution of $ B_1 $ is of the form of \eqref{II} with $ N_3\sim N_4 \gtrsim N_2 $ and thus 
using again Lemma \ref{lemtriest}, we get  exactly the same estimate as for $ D_2 $.

To bound the contribution of $ B_2 $ we use integration by parts and the commutator estimate \eqref{commu} 
 and proceed as in \eqref{cco} to get 
\begin{align}
 N^{2s}  \Bigl|  \int_{]0,t[\times \R} B_2   P_N u \Bigr|  & \lesssim N^{2s}\sum_{1\ll N_2\ll N } \Bigl|   \int_{]0,t[\times \R} \partial_x \Bigl( P_{\ll N} P_{\gtrsim N_2}e\,  u_{N_2}\Bigr)  ( P_N u)^2 \Bigr| \nonumber \\
  & +N^{2s}\sum_{1\ll N_2\ll N }  \Bigl|    \int_{]0,t[\times \R}  [P_N, P_{\ll N} P_{\gtrsim N_2}e\,  u_{N_2}] \tilde{P}_N u_x  P_N u
  \Bigr| \nonumber \\
  &\lesssim T\sum_{1\ll N_2\ll N } N^{2s} \|\partial_x(P_{\ll N} P_{\gtrsim N_2}e\, u_{N_2})\|_{L^\infty_{Tx}} 
\|\tilde{P}_N u \|_{L^\infty_T L^2_x}^2 \nonumber \\
& \lesssim \delta_N \, T  \| e\|_{L^\infty_T C_\ast^{1+}} \, \|u\|_{L^\infty_T H^{\frac{1}{2}+} } \|u\|_{L^\infty_T H^s}^2 \label{cco2}
\end{align}
with $ \|(\delta_{2^j})_{j\ge 0}\|_{l^1} \le 1 $.\vspace*{2mm} \\
$ \bullet $ {\it Contribution of $ \partial_x P_N(b u_{x}) $}.
This term being linear, we will give an estimate for $s >-1/2 $. 
Integrating by parts, the contribution of this term can be rewritten as :
$$
    \cro{N}^{2s} \int_{]0,t[\times \R}  \partial_x P_N(b u_{x}  )  P_N u  =    -  \cro{N}^{2s} \int_{]0,t[\times \R} P_N(b u_{x}  )  P_N u_x 
   $$
 For  $ 1\le N\lesssim 1 $, it then holds 
 \begin{align}
 \cro{N}^{2s} \Bigl|   \int_{]0,t[\times \R} & P_N(b u_{x}  )  P_N u_x  \Bigr| \nonumber \\
 & \lesssim     \cro{N}^{2s}\Bigl| \int_{]0,t[\times \R} P_N(\tilde{P}_N b \,   \partial_x u_{\ll N}   )  P_N u_x  \Bigr|\nonumber \\
    &\hspace*{1cm} +   \cro{N}^{2s} \Bigl|\sum_{N_1\gtrsim N}  \int_{]0,t[\times \R} \tilde{P}_{N_1} b\,  \partial_x u_{N_1}   )  P_N u_x  \Bigr|\nonumber \\
& \lesssim  T \|b\|_{L^\infty_T C_*^0} \|u\|_{L^\infty_T L^2_x}^2 +\|u\|_{L^\infty_T L^2_x} \int_0^t  \sum_{N_1\gtrsim N}N_1  \|b_{N_1}\|_{L^\infty_x} \|u_{N_1} \|_{L^2_x}\nonumber \\
&  \lesssim  T \|b\|_{L^\infty_T C_*^1} \|u\|_{L^\infty_T L^2_x}^2 
\end{align}
which is acceptable. 
For $ N\gg 1$, we decompose this term as  
 \begin{align}\label{bb}
 \cro{N}^{2s}  \int_{]0,t[\times \R} & \partial_x P_N(b u_{x}) P_N u \nonumber \\ 
& =  - \cro{N}^{2s}\int_{]0,t[\times \R}\! \! b \, (P_N u_x)^2 -  \cro{N}^{2s} \int_{]0,t[\times \R} [P_N,b] u_x P_N u_x 
 \end{align}
 The first term of the right-hand side is non positive and will give us an estimate   on the $ L^2_{[b]}(0,T;H^s) $-semi norm of $ u $. Note that  the contribution of the low frequency part of $ u $, $ N\lesssim 1$, to this semi norm is easily estimated by 
 \begin{equation}\label{bba}
 \sum_{1\le N \lesssim 1} \cro{N}^{2s}\int_{]0,t[\times \R} b (P_N u_x)^2 \lesssim \|b\|_{L^\infty_{Tx}} \|u\|_{L^\infty_T H^s}^2 \; .
   \end{equation}

  To control the second term of the right-hand side, we perform a frequency decomposition of $ b$ in the following way : 
 \begin{align}
   N^{2s} \int_{]0,t[\times \R} [P_N,b] u_x P_N u_x &=  N^{2s} \int_{]0,t[\times \R} [P_N, b_{\gtrsim N} ] u_x P_N u_x\nonumber \\
   & \hspace*{1cm}+ N^{2s}   \int_{]0,t[\times \R} [P_N, b_{\ll N}]  u_x P_N u_x \nonumber \\
   & =A+B \; .
  \end{align}
   $A $ is easily estimated by 
  \begin{align}
  |A| & \le  N^{2s} \sum_{N_1\sim N} \Bigl| \int_{]0,t[\times \R} [P_N, b_{ N_1} ] P_{\lesssim N} u_x P_N u_x \Bigr| \nonumber \\
  & \hspace{1cm}+N^{2s} \sum_{N_1\gg  N} \Bigl| \int_{]0,t[\times \R} P_N (b_{ N_1} \tilde{P}_{N_1} u_x) P_N u_x \Bigr|
 \nonumber \\
& \lesssim T N^{s+1}  N^{0\vee 1-s}  \|b_{\sim N} \|_{L^\infty_{Tx}}  \|u_x\|_{L^\infty_T H^{s-1}}^2 \nonumber \\
&\hspace{1cm}+N^{s+1}    \|u_x\|_{L^2_T H^{s-1}} \sum_{N_1\gg N} N_1^{-s-1}  \|P_{N_1} b_x \|_{L^\infty_{T}C^1_*}  \|\tilde{P}_{N_1}u_x\|_{L^2_T H^{s-1}} \nonumber \\
& \lesssim \delta_N T   \| b_x \|_{L^\infty_{T}C^{s\vee 1}_*}  \|u\|_{L^\infty_T H^s}^2
\end{align}
  that is acceptable. Finally applying \eqref{comcom} and \eqref{commu2} we easily obtain 
   \begin{equation}
  |B|  \lesssim T \|b_{xx} \|_{L^\infty_{Tx}} \| \tilde{P}_N u \|_{L^2_{T}H^s}^2  
 \lesssim \delta_N T   \|b_{xx} \|_{L^\infty_{Tx}}  \|u\|_{L^\infty_T H^s}^2 \label{estfin}
\end{equation}
Gathering \eqref{estenergy}-\eqref{estfin}, \eqref{estHsregular} follows.
 \end{proof}

\subsection{Estimate in $ H^{s-1}(\R) $ on the difference of two solutions}
   \begin{proposition} \label{prodif}
Let $ 0<T<1$ and $ u,v \in Y^s_T $ with $ s>1/2 $ be two solutions to \eqref{KdV2}  { associated with two initial data
 $ u_0,v_0\in H^s(\R) $}. Then it holds
\begin{equation}\label{estdiffHsregular}
\|u-v\|_{L^\infty_TH^{s-1}}^2 +
\|u-v\|_{L^2_{[b]}(]0,T[;H^s)}^2 \lesssim \|u_0-v_0\|_{H^{s-1}}^2  +C T^{\frac{1}{16}}\|u+v\|_{Y^s_T}
\|u-v\|_{Y^{s-1}_T}^2 \;.
\end{equation}
with 
$$
C=C\Bigl(s, \|b\|_{L^\infty_T C^{2\vee s}_*}, \|c\|_{L^\infty_T C^{(1\vee (s-1)\vee (2-s))+}_*}, 
 \|d\|_{L^\infty_T C^{|s-1|+}_*},  \|e\|_{L^\infty_T C^{(\frac{3}{2}\vee (s+\frac{1}{2}))+}_*}\Bigr)
$$
\end{proposition}
\begin{proof}
 The difference $w=u-v$ satisfies
\begin{equation}\label{eq-diff}
 w_t + w_{3x} -b w_{2x} +c w_{x} + d w = \frac{1}{2} e \partial_x(z w) +f  zw
\end{equation}
where $z=u+v$. We proceed as in the proof of the preceding proposition by applying the operator $ P_N $, with $ N\in 2^{\N} $, to the above equation, taking the $ L^2_x $ scalar product with $ P_N w$, multiplying by $ \cro{N}^{2(s-1)} $ and integrating on $ ]0,t[ $ with $ 0<t<T $. 
 Clearly the terms coming from the linear part of \eqref{KdV2} (i.e. the  term where $ z$ is not involves) may be treated by the estimates established in the proof of the preceding proposition. They lead to 
\begin{align}
\cro{N}^{2(s-1)} \Bigl| \int_{]0,t[\times \R}  P_N (dw)    P_N w \Bigr| 
& \lesssim T \delta_N \, \|d\|_{L^\infty_T C_*^{|s-1|+}}  \|w\|_{L^\infty_T H^{s-1}}^2
\end{align}
\begin{align}
\cro{N}^{2(s-1)} & \Bigl| \int_0^t\int_{\R}  \Bigl(P_N((b_x+c) w_x) \Bigr) P_N w \Bigr| \nonumber \\
& \lesssim  T \delta_N \,  (\|b_{x}\|_{L^\infty_{T}C^{(1\vee s-1\vee 2-s)+}_\ast}+ \|c\|_{L^\infty_{T}C^{(1\vee s-1\vee2-s)+}_\ast}) \|w\|^2_{L^\infty_T H^{s-1}}
\end{align}

\begin{align} \cro{N}^{2(s-1)} \int_{]0,t[\times \R}  \partial_x P_N(b w_{x}  )  P_N w \lesssim 
\|b\|_{L^\infty_T C^2_*}  \|w\|^2_{L^\infty_T H^{s-1}}
\end{align}
Therefore, proceeding  as in the proof of the preceding proposition, we infer that for $ N\ge 1 $,
\begin{align}
\|P_N w\|^2_{L^\infty_T H^{s-1}}  & \lesssim \|P_N w_0\|_{H^{s-1}}^2 +\delta_N T  \tilde{C} \|w\|^2_{L^\infty_T H^{s-1}}
\nonumber\\
& \hspace*{0.5cm}+ \sup_{t\in ]0,T[}  \langle N\rangle^{2(s-1)}\Bigl| \int_0^t\int_{\R}  P_N\Bigl(e \partial_x(zw)+f zw\Bigr) P_N w \Bigr|
\end{align}
with 
$$
 \tilde{C}=\tilde{C}\Bigl(s, \|b\|_{L^\infty_T C^{2\vee s}_*}, \|c\|_{L^\infty_T C^{(1\vee (s-1)\vee (2-s))+}_*}, 
 \|d\|_{L^\infty_T C^{|s-1|+}_*}\Bigr)
$$
To control the contribution of $P_N(fzw) $ we use Lemma \ref{product} to get 
\begin{align}
 \langle N\rangle^{2(s-1)}\Bigl| \int_0^t\int_{\R}  P_N (f zw) &  P_N w \Bigr|  \lesssim 
 \delta_N T \|fzw\|_{L^\infty_T H^{s-1}} \|w\|_{L^\infty_T H^{s-1}}\nonumber \\
 & \lesssim \delta_N T \|fz\|_{L^\infty_T H^{s}} \|w\|_{L^\infty_T H^{s-1}}^2\nonumber \\
  & \lesssim \delta_N T \|f\|_{L^\infty_T C_*^{|s|+}} \|z\|_{L^\infty_T H^{s}}\|w\|_{L^\infty_T H^{s-1}}^2
\end{align}
It remains to tackle the contribution of $ P_N\Bigl(e \partial_x(zw)\Bigr) $.
For $ 1\le N\lesssim 1 $, we write $ e\,   \partial_x(zw) = \frac{1}{2} \partial_x (ezw)-\frac{1}{2} e_x zw $ to get 
\begin{align}
N^{2(s-1)} \Bigl| \int_{]0,t[\times \R}  P_N & \Bigl( e   \partial_x(zw) \Bigr) P_N u \Bigr| \nonumber \\
&\lesssim  T \|w\|_{L^\infty_T H^{s-1}} (\| e_x z w  \|_{L^\infty_T H^{-1}} +\|e z w \|_{L^\infty_T H^{-1}} ) \nonumber \\
& \lesssim \,  T \| e\|_{L^\infty_{T}C^{\frac{3}{2}+}_\ast} \, \|z\|_{L^\infty_T H^{\frac{1}{2}+}} \|w\|_{L^\infty_T H^{-\frac{1}{2}+}}^2
\end{align}
since $ s+s-1>0 $.\\
It thus remains to consider $ N\gg 1$.
 Because of the lack of symmetry with respect to the estimate on $ u $,  we  consider this time three different contributions. \\
{\bf 1.} {\it The contribution of} $ P_N(P_{\gtrsim N} e \, \partial_x(zw)) $.  This contribution is easily estimated by 
\begin{align}
N^{2(s-1)} \Bigl| \int_{]0,t[\times \R} & P_N \Bigl( P_{\gtrsim N} e\,   \partial_x(zw) \Bigr) P_N w \Bigr| \nonumber \\ 
&=  N^{2(s-1)}   \Bigl| \int_{]0,t[\times \R} P_N \Bigl( P_{\sim N} e \, P_{\ll N}\partial_x(zw) \Bigr) P_N w \Bigr| \nonumber \\
& +  N^{2(s-1)} \sum_{N_1\gtrsim N}  \Bigl| \int_{]0,t[\times \R} P_N \Bigl( P_{\sim N_1} e \, P_{N_1}\partial_x(zw) \Bigr) P_N w \Bigr|\nonumber \\
& \lesssim N^{2(s-1)}\int_{0}^t  \| P_{\sim N}e \|_{L^\infty_x} \|P_N w\|_{L^2_x} N^{3/2}\|zw \|_{H^{-1/2}} \nonumber \\
&+ N^{2(s-1)}\int_{0}^t  \|P_N w\|_{L^2_x}\sum_{N_1\gtrsim N} \| P_{\sim N_1} e \|_{L^\infty_x}  N_1^{2-s}
 \|\partial_x(zw )\|_{H^{s-2}} \nonumber \\
& \lesssim \delta_N \,  T  \| e\|_{L^\infty_T C_\ast^{((2-s)\vee(s+\frac{1}{2}))+}} \, \|z\|_{L^\infty_T H^{s}}
 \|w\|_{L^\infty_T H^{s-1}}^2
\end{align}
since for $ s>1/2$, $((2-s)\vee 1\vee (s+1/2) =(2-s)\vee(s+\frac{1}{2})$.  .\vspace{2mm}\\
\noindent
{\bf 2.} {\it The contribution of} $ P_N(P_{\ll N} e  \, z_xw) $. We rewrite this term as 
\begin{align}
 P_N(P_{\ll N} e\,  z_x w )   = 
 & P_N\Bigl( P_{\ll N} e  \, P_{\lesssim 1}  w \,  \tilde{P}_N z_x\Bigr) +P_N\Bigl( P_{\ll N}e \,  \tilde{P}_N w \,  P_{\lesssim 1}z_x\Bigr) \nonumber \\
  & + \sum_{  1\ll N_3,  N_2} P_N\Bigl( P_{\ll N_2\wedge N_3} P_{\ll N}e \, w_{N_2} \partial_x z_{N_3} \Bigr)\nonumber \\
  & + \sum_{  1\ll N_3,  N_2} P_N\Bigl( P_{\gtrsim N_2\wedge N_3} P_{\ll N}  e \, w_{N_2} \partial_x z_{N_3} \Bigr)
  \nonumber \\
  =& A+B+C+D. \label{decou}
\end{align}
Proceeding as in the proof of \eqref{td1}, it is not too difficult to check that the contributions of $ A $ and $ B $ can be bounded by 
\begin{equation}
N^{2(s-1)} \Bigl| \int_{]0,t[\times \R}  (A+B) P_N w \Bigr| \lesssim T \delta_N \|e\|_{L^\infty_{Tx}} \|z\|_{L^\infty_T H^s} \|w\|_{L^\infty_T H^{s-1}}^2 \; .
\end{equation}
To bound  the contribution of $ C $ , we notice that  the integral is of the form \eqref{II} so that we can use Lemma \ref{lemtriest}. Proceeding as in \eqref{E1}-\eqref{E2} we  get 
\begin{equation}
 N^{2(s-1)} \Bigl| \int_{]0,t[\times \R}  C P_N w \Bigr| \lesssim T^\frac{1}{16}   \delta_N  \;  (\|e\|_{L^\infty_{Tx}}+\|e_t \|_{L^\infty_{Tx}}) \|z\|_{Y^{s}_T}\|w\|_{Y^{s-1}_T}^2 \; 
 \end{equation}
Finally we rewrite $ D $ as 
$$
D=\sum_{    N_2\gg 1 } P_N\Bigl( P_{\gtrsim N_2} P_{\ll N}  e \, w_{N_2} \tilde{P}_N z_x \Bigr)
+\sum_{ N_3\gg 1} P_N\Bigl( P_{\gtrsim  N_3} P_{\ll N}  e \, \tilde{P}_N w\partial_x z_{N_3} \Bigr)
$$
Proceeding as in \eqref{d1} we easily get 
\begin{align}
N^{2(s-1)} \Bigl| \int_{]0,t[\times \R} D P_N w \Bigr| \lesssim  T \delta_N   \| e\|_{L^\infty_T C_\ast^{1}} \|z\|_{L^\infty_T H^s} \|w\|_{L^\infty_T H^{s-1}}
\end{align}
{\bf 3.} {\it The contribution of} $ P_N(P_{\ll N} e \, (z w_x)) $.  We rewrite this term as 
\begin{align}
 P_N(P_{\ll N} e\,  z w_x )   = 
 & P_N\Bigl( P_{\ll N} e  \,  \tilde{P}_N z P_{\lesssim 1}  w_x\Bigr) +P_N\Bigl( P_{\ll N}e \,  P_{\lesssim 1}z \,  w_x \Bigr) \nonumber \\
  & + \sum_{  1\ll N_3,  N_2} P_N\Bigl( P_{\ll N_2\wedge N_3} P_{\ll N}e \, z_{N_2} \partial_x w_{N_3}  \Bigr)\nonumber \\
  & + \sum_{  1\ll N_3,  N_2} P_N\Bigl( P_{\gtrsim N_2\wedge N_3} P_{\ll N}  e \, z_{N_2} \partial_x w_{N_3} \Bigr)
  \nonumber \\
  =& \tilde{A}+\tilde{B}+\tilde{C}+\tilde{D}.
\end{align}
Proceeding as in \eqref{td1}, we easily get 
\begin{equation}
N^{2(s-1)} \Bigl| \int_{]0,t[\times \R}  \tilde{A}P_N w \Bigr| \lesssim T \delta_N \|e\|_{L^\infty_{Tx}} \|z\|_{L^\infty_T H^s} \|w\|_{L^\infty_T H^{s-1}}^2 \; .
\end{equation}
To bound the contribution of $\tilde{B} $ we proceed as in \eqref{cco},  integrating by parts and using the commutor estimate \eqref{commu} to get 
\begin{align}
 N^{2(s-1)} \Bigl| \int_{]0,t[\times \R} \tilde{B} P_N u \Bigr|
& \lesssim \delta_N \, T  \| e\|_{L^\infty_T C_\ast^{1+}} \, \|z\|_{L^\infty_T L^2_x } \|w\|_{L^\infty_T H^{s-1}}^2 \label{tildeB}
\end{align}
{\it Finally the contributions of} $\tilde{C} $ and $\tilde{D} $ can be estimated exactly as the ones of $C $ and $D$.
\end{proof}
\begin{remark}\label{rem41}
Gathering Lemma \ref{estYs}  and Propositions \ref{prou}-\ref{prodif} we observe that sufficient hypotheses for these statements to hold are 
\begin{equation} \label{hypo2}
 \begin{array}{l} 
b\in L^\infty_T C_*^{((s+1)\vee 2)+} , \quad c\in  L^\infty_T C_*^{((2-s)\vee s)+}, \quad  d\in  L^\infty_T C_*^{|s|+}\\
\quad e\in  L^\infty_T C_*^{((s+\frac{1}{2})\vee \frac{3}{2})+}
, \quad e_t \in L^\infty_{Tx} \quad\text{and}\quad  f\in  L^\infty_T C_*^{|s|+} 
\end{array} 
\end{equation}
\end{remark}
\section{Proof of Theorem \ref{th2}}\label{sect5}
\subsection{Uniqueness}\label{subsect51}
Assume \eqref{hypo2} are fulfilled and $ u_0\in H^s(\R) $ with $ s>1/2 $. Let $ u $ and $v $ be two solutions of \eqref{KdV2} emanating from $ u_0 $ that belong to $ L^\infty_T H^s\cap  L^2_{[b]}(]0,T[;H^{s+1})$ for some $ T>0 $. Then according to Lemma \ref{estYs}, $u $ and $ v$ belong to $ Y^s_T $ and Proposition \ref{prodif} together with \eqref{estdiffXregular}  ensure that for any $0< T_0 \le T\wedge 2 $ it holds 
\begin{align*}
&\|u-v\|_{L^\infty_{T_0} H^{s-1}}^2 +
\|u-v\|_{L^2_{[b]}(]0,T_0[;H^s)}^2 \\
&  \hspace*{5mm} \lesssim  T_0^{\frac{1}{16}}
(1+\|u+v\|_{Y^s_T})^3
\Bigl( \|u-v\|_{L^\infty_{T_0} H^{s-1}}^2 +
\|u-v\|_{L^2_{[b]}(]0,T_0[;H^s)}^2\Bigr) \;.
\end{align*}
This forces $ u\equiv v $ on some time interval $]0,T_1[$ with $ 0<T_1\le T_0 $. Taking now $ T_1$ as initial time we can repeat the same argument to get that $ u\equiv v $ on $]0,T\vee 2 T_1[ $ and a finite iteration of this argument leads to $  u\equiv v $ on $]0,T[ $. It is worth noticing that in the case $ b\equiv 0 $, $ L^\infty_T H^s\cap  L^2_{[b]}(]0,T[;H^{s+1})= L^\infty_T H^s$ and thus we get the unconditional uniqueness of \eqref{KdV2} in $ H^s(\R) $ for $s>1/2$. 
\subsection{Existence}
We make use of the famous existence result of  Craig-Kappeler-Strauss \cite{CKS} for the  general quasilinear KdV type equations :
\begin{equation}
u_t + F(\partial_x^3 u, \partial_x^2 u , \partial_x u, u,x,t)=0  \label{eqCKS} \; .
\end{equation}
In this paper, the following assumptions on $ F$ are made : \\
$ F\, :\, \R^5\times [0,T] \to \R $ is $ C^\infty $ in all its variables and satisfies
\begin{enumerate}
\item[(A1)] $\exists c>0 $ such that $ \partial_1 F(y,x,t)\ge c>0 $ for all $ y\in \R^4 $, $x\in\R $ and $ t\in [0,T] $. 
\item[(A2)]$\partial_2 F(y,x,t) \le 0 $.
\item[(A3)] All the derivatives of $ F(y,x,t) $ are bounded for $x\in\R $, $t\in [0,T] $ and $ y $ in a bounded set.
\item[(A4)] $x^N\partial_x^j F(0,x,t) $ is bounded for all $ N\ge 0 $, $j\ge 0 $, $x\in \R $ and $ t\in (0,T] $.
\end{enumerate}
Fixing  $ F$  that satisfies (A1)-(A4), in \cite{CKS} it is shown that for any $ k\in \N $ with $  k\ge 7 $ and any $ c_0>0 $ there  exists $ T=T(c_0)>0 $ such that for any $ u_0\in H^k(\R) $,  with 
 $\|u_0\|_{H^7} \le c_0 $, the Cauchy problem associated with \eqref{eqCKS} has a unique local solution $ u\in L^\infty(0,T; H^k(\R)) $. 
 
 This implies that for any $ F $ satisfying (A1)-(A4) and any $ u_0 \in H^k $ with $k\ge 7 $, the unique solution  $ u $ to \eqref{eqCKS} can be prolonged on a maximal  time interval $ [0,T^*[ $ with either 
 \begin{equation}\label{alt}
 T^*=+\infty \quad \text{ or } \quad \limsup_{T\nearrow T^*} \|u\|_{L^\infty(0,T;H^7)} =+\infty \; .
 \end{equation}
We notice that \eqref{KdV2} corresponds to \eqref{eqCKS} with 
$$
 F(y,x,t)=y_1-b(t,x) y_2 +c(t,x)y_3+d(t,x) y_4-e(t,x) y_3 y_4 -f(t,x) y_4^2
 $$
 In particular, for any $y\in\R^4$, $x\in\R $ and $ t\in [0,T] $  we have $ \partial_1 F(y,x,t)= 1 $ and $F(0,x,t)=0 $ which ensure that 
 $(A1) $ and $(A4) $  are clearly fulfilled. Moreover, the hypothesis $ b\ge 0 $ ensures that $(A2) $ is also fulfilled. Therefore,  since our coefficient functions are by hypothesis all bounded on $ [0,T] \times \R $, it thus suffice  to regularize them by convoluting in $(t,x) $ with a smooth positive sequence of mollifiers to fulfill the assumptions (A1)-(A4).

So let the coefficient functions $ a,b,c,d,e,f $ satisfying the hypotheses of Theorem  \ref{th2} and let $ u_0 \in H^{s}(\R) $ with $ s>1/2$. We first construct the solution  emanating from $ u_0 $ to \eqref{KdV2}  with $ a,b,c,d,e $ replaced by their smooth regularizations. For this we regularize the initial datum by setting, for any $ n\in \N^*$, 
 $ u_{0,n}=P_{\le n} u_0  \in H^\infty(\R) $. According to the existence result of \cite{CKS} there exists a sequence $ (T_n) $  with  $0<T_n<1 $  such that, for any $ n\in \N^*$, \eqref{KdV2} has a unique solution $ u_n\in L^\infty(0,T_n;H^{\infty}(\R)) $ emanating from $u_{0,n} $. 
 Note that \eqref{KdV2} then implies that actually $u_n\in C([0,T_n];H^\infty(\R)) $.
Now, applying \eqref{esta1} and \eqref{estHsregular} for $ u_n $ on $[0, T_n] $ we obtain  that 
\begin{align*}
 \|u_n\|_{L^\infty_{T_n} H^{s_0}}^2 +& \|u_n\|_{L^2_{[b]}(]0,T_n[;H^{s_0+1})}^2 \\
& \le \|u_0\|_{H^s}^2 + C \; T_n^\frac{1}{16} \Bigl(1+
\|u\|_{L^\infty_{T_n} H^{s_0}}^2 + \|u\|_{L^2_{[b]}(]0,T_n [;H^{s_0+1})}\Bigr)^6 \;
\end{align*}
for $ s_0=\frac{1}{2}+< s $.
Using the continuity of 
$ T\to \|u_n\|_{L^\infty(0,T;H^{s_0})}+\|u\|_{L^2_{[b]}(]0,T [;H^{s_0+1})}$
this ensures that  there exists $ 0<T_0=T_0(\|u_0\|_{H^{s_0}})<2 $ such that  
$$\|u_n\|_{L^\infty(0,T_2;H^{s_0} )}+\|u\|_{L^2_{[b]}(]0,T_2 [;H^{s_0+1})}
 \le 4 \|u_{0,n}\|_{H^{s_0} }\quad \text{ for }\quad T_2=T_n\wedge T_0\; .
$$
Using again  \eqref{esta1} and \eqref{estHsregular},  we obtain that, for any fixed $ n\ge 0 $, 
 $ u_n $ is bounded in $ L^\infty_{T_2} H^7 $. Therefore 
 \eqref{alt} ensures that $ u_n $ can be extended on $ [0, T_0]$. Hence, it holds 
 $$\|u_n\|_{L^\infty(0,T_0;H^{s_0})}+\|u_n\|_{L^2_{[b]}(0,T_0;H^{s_0+1})}\le 4 \|u_{0}\|_{H^{s_0}} \; .
 $$
 Applying again   \eqref{esta1} and \eqref{estHsregular} but at the $ H^s $-regularity  this forces
 $$
 \|u_n\|_{L^\infty(0,T_0;H^{s})}+\|u_n\|_{L^2_{[b]}(0,T_0;H^{s+1})}\lesssim  \|u_{0}\|_{H^{s}}\; .
 $$
 Note that Lemma \ref{estYs} and Proposition \ref{prodif} then ensure that $ (u_n) $ is a Cauchy sequence in $L^\infty(0,T_0;H^{s-1}) $ and thus it is  also a Cauchy sequence in 
  $L^\infty(0,T_0;H^{\frac{1}{2}+}) $. Let $ u $ be the limit of $ u_n $ in  $L^\infty(0,T_0;H^{\frac{1}{2}+}) $. From the above estimates we know that 
  $ u\in L^\infty(0,T_0; H^s) $ and  it is immediat to check that 
   $ u $ satisfies \eqref{KdV2} at least in $ L^\infty(0,T_0; H^{s-3}) $.
      
   Now we can pass to the limit on the coefficient functions. Since their regularizations are bounded in the function spaces appearing in Remark \ref{rem41}, we obtain the existence of a solution 
    $ u\in L^\infty(0,T_0;H^s)\cap L^2_{[b]}(0,T_0;H^{s+1})$ that is the unique one in this class on account of  Subsection \ref{subsect51}.
    Now the continuity of $ u $ with values in $H^{s}(\R)$ as well as the continuity of the flow-map in $H^{s}(\R)$ will follow
from the Bona-Smith argument (see \cite{BS}).
 For any $ \varphi\in H^s(\R) $, any integer $ n \ge 1$ and any $ r\ge 0 $,  straightforward calculations in Fourier space lead to
   \begin{equation}\label{init}
   \|P_{\le n} \varphi\|_{H^{s+r}_x}\lesssim n^r\|\varphi\|_{H^s_x}\quad\mbox{and} \quad
   \| \varphi-P_{\le n} \varphi\|_{H_x^{s-r}} \lesssim n^{-r} \| P_{>n} \varphi\|_{H^s_x}\; .
   \end{equation}
   %%%%%%%%%%%%%%%%%%%%%%%%
   Let $ u_0 \in H^s $ with $ s>1/2$ and let $T_0=T_0(\|u_0\|_{H^{\frac{1}{2}+}}) >0 $ the associated minimum time of existence. We denote by $ u_n \in L^\infty(0,T_0;H^s) $ the solution of \eqref{KdV2} emanating from $
   u_{0,n}=P_{\le n} u_0$ and
for $1\le n_1 \le n_2  $, we set
$$
w:=u_{n_1}-u_{n_2} \, .
$$
{Then,   \eqref{estdiffHsregular}-\eqref{estdiffXregular}  lead   to
\begin{equation}\label{to1}
\|w\|_{Y_{T_0}^{s-1}}\lesssim \|w(0)\|_{H^{s-1}} \lesssim n_1^{-1} \|P_{>n_1} u_0\|_{H^s}\, .
\end{equation}
  Moreover, for any $ r\ge 0 $ and $ s>1/2 $ we have
\begin{equation}\label{to2}
\|u_{n_i}\|_{Y_{T_0}^{s+r}}
\lesssim
\|u_{0,n_i}\|_{H^{s+r}}
\lesssim n_i^{r} \|u_0\|_{H^s} \, .
\end{equation}
}
Next, we observe that $w$ solves the equation
\begin{equation}\label{W}
 w_t + w_{3x} -b w_{2x} +c w_{x} + d w = \frac{1}{2} e \partial_x(w^2) + e \partial_x(u_{n_1} w)+f  w^2 +2 f u_{n_1} w\, .
\end{equation}
\begin{proposition}\label{pro3}
Let $ 0<T<1$ and $ w\in Y^s_T $ with $ s>1/2 $ be a solution to \eqref{W}. Then it holds
\begin{eqnarray}
\|w\|_{L^\infty_TH^s}^2  &\lesssim & \|w(0)\|_{H^s} ^2 +C \, T^\frac{1}{16} \Bigl( ( \|u_{n_1} \|_{Y_T^s} 
+  \|u_{n_2} \|_{Y_T^s}) \|w\|_{Y_T^s}^2
 \nonumber \\
& &   + \|u_{n_1} \|_{L^\infty_T H^{s+1}} \|w\|_{L^\infty_T H^{s-1}} \|w\|_{L^\infty_T H^{s}}\Bigr)
\;.
\label{estW}
\end{eqnarray}
\end{proposition}
\begin{proof}
 It is a consequence  of  estimates derived in the proof of Propositions \ref{prou} and \ref{prodif}.
Actually because of the loss of symetry we only have to take care of the contribution of $P_N(P_{\ll N }e \partial_x u_{N_1} w) $. We decompose this term as in \eqref{decou} to get 
\begin{align}
 P_N(P_{\ll N} e\,  \partial_x u_{n_1} w )   = 
 & P_N\Bigl( P_{\ll N} e  \, P_{\lesssim 1}  w \,  \tilde{P}_N  \partial_x  u_{n_1}\Bigr) +P_N\Bigl( P_{\ll N}e \,  \tilde{P}_N w \,  P_{\lesssim 1} \partial_x u_{n_1}\Bigr) \nonumber \\
  & + \sum_{  1\ll N_3,  N_2} P_N\Bigl( P_{\ll N_2\wedge N_3} P_{\ll N}e \, w_{N_2}P_{N_3} \partial_x u_{n_1} \Bigr)\nonumber \\
  & + \sum_{  1\ll N_3,  N_2} P_N\Bigl( P_{\gtrsim N_2\wedge N_3} P_{\ll N}  e \, w_{N_2} P_{N_3}\partial_x u_{n_1} \Bigr)
  \nonumber \\
  =& A+B+C+D. \label{decou2}
\end{align}
The contribution of $ A$ and $ B $  can be  easily estimated by 
\begin{equation}
N^{2s} \Bigl| \int_{]0,t[\times \R}  A P_N w \Bigr| \lesssim T \delta_N \|e\|_{L^\infty_{Tx}} \|u_{n_1}\|_{L^\infty_T H^{s+1}}
 \|w\|_{L^\infty_T H^{-1/2}}   \|w\|_{L^\infty_T H^{s}} \; .
\end{equation}
and
\begin{equation}
N^{2s} \Bigl| \int_{]0,t[\times \R}  B P_N w \Bigr| \lesssim T \delta_N \|e\|_{L^\infty_{Tx}} \|u_{n_1}\|_{L^\infty_T L^2}
  \|w\|_{L^\infty_T H^{s}}^2 \; .
\end{equation}
To bound the contribution of $C $ we use again Lemma \ref{lemtriest} and proceed as in \eqref{E1}-\eqref{E2}  to get 
\begin{equation}
 N^{2s} \Bigl| \int_{]0,t[\times \R}  C P_N w \Bigr| \lesssim T^\frac{1}{16}   \delta_N  \;  (\|e\|_{L^\infty_{Tx}}+\|e_t \|_{L^\infty_{Tx}}) \|u_{n_1}\|_{Y^{s}_T}\|w\|_{Y^{s}_T}^2 \; 
 \end{equation}
 Finally we rewrite $ D $ as 
\begin{align*}
D& =\sum_{    N_2\gg 1 } P_N\Bigl( P_{\gtrsim N_2} P_{\ll N}  e \, w_{N_2} \tilde{P}_N \partial_x u_{n_1}  \Bigr)
+\sum_{ N_3\gg 1} P_N\Bigl( P_{\gtrsim  N_3} P_{\ll N}  e \, \tilde{P}_N w\partial_x u_{n_1} \Bigr)\\
&= D_1+D_2 
\end{align*}
We easily get 
\begin{align}
N^{2s} \Bigl| \int_{]0,t[\times \R} D_1 P_N w \Bigr|  & \lesssim  \delta_N  \int_0^t 
\sum_{N_2\gg 1} \| P_{\gtrsim N_2} e\|_{L^\infty_x} \|w_{N_2} \|_{L^\infty_x} \|u_{n_1}\|_{H^{s+1}}\|w\|_{H^s}\nonumber \\
& \lesssim  \delta_N  \int_0^t 
\sum_{N_2\gg 1} \| P_{\gtrsim N_2} e\|_{L^\infty_x}  N_2^{\frac{3}{2}-s}\|w_{N_2} \|_{H^{s-1}} 
\|u_{n_1}\|_{H^{s+1}}\|w\|_{H^s}\nonumber\\
& \lesssim  \delta_N  T 
 \|  e\|_{L^\infty_x C^1_*} 
\|u_{n_1}\|_{L^\infty_T H^{s+1}}  \|w \|_{L^\infty_T H^{s-1}} \|w\|_{L^\infty_T H^s}
\end{align}
since $ s>1/2$. In the same way we get 
\begin{align}
N^{2s} \Bigl| \int_{]0,t[\times \R} D_2 P_N w \Bigr|  & \lesssim  \delta_N  \int_0^t 
\sum_{N_3\gg 1} \| P_{\gtrsim N_3} e\|_{L^\infty_x} \|\partial_x P_{N_3} u_{n_1}\|_{L^\infty_x} \|w\|_{H^s}^2\nonumber \\
& \lesssim  \delta_N  \int_0^t 
\sum_{N_3\gg 1} \| P_{\gtrsim N_3} e\|_{L^\infty_x}  N_3^{\frac{3}{2}-s}\| P_{N_3} u_{n_1}\|_{H^{s}}
\|w\|_{H^s}^2\nonumber\\
& \lesssim  \delta_N  T 
 \|  e\|_{L^\infty_x C^1_*} 
\|u\|_{L^\infty_T H^{s}}\|w\|_{L^\infty_T H^s}^2
\end{align}
that completes the proof of the proposition.
\end{proof}
{ Combining \eqref{estdiffXregular} with  \eqref{estW} and \eqref{to2} we get for $0<T<T_0$, 
 \begin{eqnarray*}
\|w\|_{Y_T^s}^2  &\lesssim & \|w(0)\|_{H^s} ^2 +T^\frac{1}{16} \Bigl[\|u_0\|_{H^s} \|w\|_{Y^s_T}^2  \nonumber \\
& & 
  + n_1 \|u_0\|_{H^s}  \|w\|_{Y_T^s} \|w\|_{Y_T^{s-1}}\Bigr]
\;.
\end{eqnarray*}
 \arraycolsep2pt
 Therefore, for $T>0 $ small enough,  \eqref{to1} leads to
\begin{eqnarray}
\|w\|_{Y_{T_0}^s} ^2 &\lesssim &  \|w(0)\|_{H^s}^2  + n_1^2   \|w\|_{Y_{T_0}^{s-1}}^2 \label{yyy}\\
 & \lesssim & \|P_{>n_1} u_0\|_{H^s}^2 \to 0 \hbox{ as } n_1\to 0 \, . \nonumber
\end{eqnarray}}
 \arraycolsep4pt
This shows that $ \{u_n\} $ is a Cauchy sequence in $C([0,T];H^s) $ and thus $ \{u_n\} $ converges in $C([0,T];H^s) $ to a solution of \eqref{KdV2} emanating from $ u_0$.Then, the uniqueness result ensures that $ u\in C([0,T];H^s)$. Repeating this argument with $u(T) $ as initial data we obtain that  $u\in C([0, T_1];H^s)$  with $ T_1=\max(2T,T_0) $. This leads to 
 $u\in C([0, T_0];H^s)  $ after finite number of repetitions. \vspace*{2mm}\\
{\bf Continuity of the flow map.} \label{cont}
Let now  $ \{u_{0}^k\}\subset
H^{s}(\R) $ be such that
$
u_{0}^k\rightarrow u_0
$
in $ H^{s}(\R) $.
We want to prove that the emanating solution $ u^k $ tends to $ u
$ in $ C([0,T_0];H^{s}) $. By the triangle inequality, for $ k $ large enough, 
$$
\|u-u^k\|_{L^\infty_{T_0}  H^{s}} \le
\|u-u_n\|_{L^\infty_{T_0}  H^{s}}+
\|u_n-u_n^k\|_{L^\infty_{T_0}  H^{s}}+
\|u^k_n-u^k\|_{L^\infty_{T_0}  H^{s}} \quad  .
$$
Using the estimate \eqref{yyy} on the solution to \eqref{W} we first infer that
$$
\|u-u_n\|_{Y_{T_0} ^{s}}+\|u^k -u^k_n \|_{Y_{T_0} ^s} \lesssim \|P_{>n} u_0\|_{H^s}+\|P_{>n} u_{0}^k\|_{H^s}
$$
and thus
\begin{equation}\label{kak1}
\lim_{n\to \infty} \sup_{k\in\N} \Bigl( \|u-u^k\|_{L^\infty_{T_0} H^{s}}+\|u^k -u^k_n \|_{L^\infty_{T_0}  H^s}\Bigr) =0 \, .
\end{equation}
 Next,  we notice that     \eqref{estdiffHsregular}-\eqref{estdiffXregular}  ensure that
{ $$
\|u_n-u_n^k\|_{Y_{T_0} ^{s-1}} \lesssim \|u_{0,n}-u_{0,n}^k\|_{H^{s-1}}
$$
and thus \eqref{yyy} and \eqref{to1} lead to
\begin{eqnarray}
\nonumber
\|u_n-u_n^k\|_{Y^s_{T_0} }^2 & \lesssim &
\|u_{0,n}-u_{0,n}^k\|_{H^{s}}^2+n^2  \|u_{0,n}-u_{0,n}^k\|_{H^{s-1}}^2\\
& \lesssim & \|u_0-u_{0}^k\|_{H^{s}} ^2(1+n^2)\, . \label{kak3}
\end{eqnarray}
}
Combining (\ref{kak1}) and (\ref{kak3}), we obtain
 the continuity of the flow map.
\section{Appendix}
\subsection{Proof of Lemma \ref{commutator}}
  We start by proving  \eqref{commu}.
Let $ N>0$.  We follow \cite{david}. By Plancherel and the mean-value theorem, 
 \begin{align*}
\Bigl| ( [P_N, P_{\ll N}f] g)(x)\Bigr| &=\Bigl|( [P_N, P_{\ll N}f] \tilde{P}_N g)(x)\Bigr| \\
 & =\Bigl|  \int_{\R} {\mathcal F}^{-1}_x(\varphi_N)(x-y) P_{\ll N}f(y)  \tilde{P}_N g(y) \, dy \\
 &\hspace*{1cm} -
 \int_{\R}  P_{\ll N}f(x)  {\mathcal F}^{-1}_x(\varphi_N)(x-y) \tilde{P}_N g(y) \, dy\Bigr| \\
 & = \Bigl|  \int_{\R}(P_{\ll N}f (y) -P_{\ll N}f(x)) N  {\mathcal F}^{-1}_x(\varphi)(N(x-y))  \tilde{P}_N g(y) \, dy\Bigr| \\
& \le \|P_{\ll N}f_x\|_{L^\infty_x}\int_{\R} N |x-y| |{\mathcal F}^{-1}_x(\varphi)(N(x-y))| |\tilde{P}_N g(y)| \, dy 
  \end{align*}
  Therefore, since $N |\cdot| |{\mathcal F}^{-1}_x(\varphi)(N \cdot)|=|{\mathcal F}^{-1}_x(\varphi')(N \cdot) | $ we deduce from Young's convolution inequalities that 
 \begin{align*} 
 \Bigl\| [P_N, P_{\ll N}f] g)\|_{L^2} & \lesssim  N^{-1} \|P_{\ll N}f_x\|_{L^\infty_x}  \|\tilde{P}_N g\|_{L^2} \; .   \end{align*}
  To prove \eqref{commu2} we proceed in the same way. We first notice that 
   \begin{align*}
I_N(x)&=( [P_N,[P_N, P_{\ll N}f]] g)(x) =( [P_N,[P_N, P_{\ll N}f]] \tilde{P}_N g)(x) \\
 & = \int_{\R^2} {\mathcal F}^{-1}_x(\varphi_N)(x-y)  {\mathcal F}^{-1}_x(\varphi_N)(y-z) 
 \Bigl(P_{\ll N}f(z)-P_{\ll N}f(y)\Bigr)  \tilde{P}_N g(z) \, dy\, dz  \\
 &- \int_{\R^2} {\mathcal F}^{-1}_x(\varphi_N)(x-y)  {\mathcal F}^{-1}_x(\varphi_N)(y-z) 
 \Bigl(P_{\ll N}f(y)-P_{\ll N}f(x)\Bigr)  \tilde{P}_N g(z) \, dy\, dz  \\
  & = \int_{\R^2} {\mathcal F}^{-1}_x(\varphi_N)(x-y)  {\mathcal F}^{-1}_x(\varphi_N)(y-z) (z-y) 
 P_{\ll N}f_x(\alpha_{y,z})  \tilde{P}_N g(z) \, dy\, dz  \\
 &- \int_{\R^2} {\mathcal F}^{-1}_x(\varphi_N)(x-y)  {\mathcal F}^{-1}_x(\varphi_N)(y-z) (y-x)
 P_{\ll N}f_x(\alpha_{y,x})\Bigr)  \tilde{P}_N g(z) \, dy\, dz  
  \end{align*}
with $ \alpha_{y,z}\in [y,z] $ and $\alpha_{y,x} \in [y,x]$. Performing the change of variable $ \theta=x+z-y $ in the last integral 
   we get 
  \begin{align*}
I_N(x) =
 \int_{\R^2} {\mathcal F}^{-1}_x(\varphi_N)(x-y)  {\mathcal F}^{-1}_x(\varphi_N)(y-z) & (z-y)
\Bigl(( P_{\ll N}f_x(\alpha_{y,z}) \\ 
&-P_{\ll N}f_x(\alpha_{x,x+z-y})\Bigr)  \tilde{P}_N g(z) \, dy\, dz  \\
  \end{align*}
with $\alpha_{x,x+z-y}\in [x, x+z-y] $. Finally, noticing that 
$$
|\alpha_{y,z}-\alpha_{x,x+z-y}|\le \max(|x-y|,|x-z|,x+z-2y |) \le 2 \max(|x-y|, |y-z|) 
$$
 and using again the mean-value theorem we eventually obtain 
  \begin{align*}
| I_N| & \le 2 \|P_{\ll N}f_{xx}\|_{L^\infty_{x}}\Bigl[ \\
&\hspace*{0.5cm} \int_{\R^2}  |z-y|^2  N^2 |{\mathcal F}^{-1}_x(\varphi)(N(z-y))|  |{\mathcal F}^{-1}_x(\varphi)(N(x-y))|
  |\tilde{P}_N g(z)| \, dy\, dz \\
  & +\int_{\R^2} |x-y| N |{\mathcal F}^{-1}_x(\varphi)(N(x-y))|  |z-y| N |{\mathcal F}^{-1}_x(\varphi)(N(z-y))|
  |\tilde{P}_N g(z)| \, dy\, dz\Bigr] 
\end{align*}
which yields to the desired result for the same reasons as above.

Finally, to prove \eqref{comcom} we first use Parseval identity  and the fact that $ g $ is real-valued  to obtain
\begin{align*}
  \int_{\R} & [P_N,P_{\ll N}f] g \; P_N g \\ 
  & = \int_{\R^2} (\varphi_N(\xi_1+\xi_2) -\varphi_N(\xi_2))
  \widehat{P_{\ll N}f}(\xi_1) \hat{g}(\xi_2)\varphi_N(\xi_1+\xi_2) \hat{g}(-\xi_1-\xi_2)\, d\xi_1 \, d\xi_2\; .
\end{align*}

Performing the change of variable $(\breve{\xi}_1,\breve{\xi}_2)=(\xi_1,-\xi_1-\xi_2) $ and recalling that $ \varphi_N $ is an even real valued function  we then get 
 \begin{align*}
  \int_{\R} & [P_N,P_{\ll N}f] g \; P_N g \\
  & =  \int_{\R^2} (\varphi_N(\breve{\xi}_2) -\varphi_N(\breve{\xi}_1+\breve{\xi}_2))
  \widehat{P_{\ll N}f}(\breve{\xi}_1) \hat{g}(-\breve{\xi}_1-\breve{\xi}_2)\varphi_N(\breve{\xi}_2) \hat{g}(\breve{\xi}_2)\, d\breve{\xi}_1 \,
   d\breve{\xi}_2\\
   & = -  \int_{\R} [P_N,P_{\ll N}f] g \; P_N g \\
   &\quad  +\int_{\R^2} (\varphi_N(\breve{\xi}_2) -\varphi_N(\breve{\xi}_1+\breve{\xi}_2))^2
  \widehat{P_{\ll N}f}(\breve{\xi}_1) \hat{g}(-\breve{\xi}_1-\breve{\xi}_2) \hat{g}(\breve{\xi}_2)\, d\breve{\xi}_1 \,
   d\breve{\xi}_2\\
      & = -  \int_{\R} [P_N,P_{\ll N}f] g \; P_N g +\int_{\R} \Bigr[P_N, [P_N,P_{\ll N}f] \Bigr] g \, g \;.
\end{align*}
This yields \eqref{comcom} by noticing that $ g $ can be replaced by $ \tilde{P}_N g $ without changing the value of  
$   \int_{\R} [P_N,P_{\ll N}f] g \; P_N g$.

\end{document}